\documentclass[12pt]{article}   

 \usepackage[T1]{fontenc}
      \input{dehyphtex}              
 \usepackage[width=16cm, height=22cm]{geometry}   
 \usepackage{amsmath,amssymb}         
 \usepackage[amsmath,thmmarks,hyperref]{ntheorem} 
 \usepackage{icomma}             
 \usepackage{enumerate}             
 \usepackage{mathrsfs}         
 \usepackage{slashed}
 \usepackage{color}
 \usepackage{esint}
 \usepackage{tikz-cd}                     

 \numberwithin{equation}{section}
 
 \usepackage[numbers,square]{natbib}
\theoremstyle{nonumberplain}  
\theoremheaderfont{\itshape}  
\theorembodyfont{\normalfont}  
\theoremseparator{.}  
\theoremsymbol{$\Box$}
\newtheorem{proof}{Proof} 
  
\theoremstyle{plain}  
\theoremheaderfont{\normalfont\bfseries}  
\theorembodyfont{\itshape}  
\theoremsymbol{~}
\theoremseparator{.}  
\newtheorem{proposition}{Proposition}[section]  
\newtheorem{corollary}[proposition]{Corollary}  
\newtheorem{lemma}[proposition]{Lemma}  
\newtheorem{theorem}[proposition]{Theorem}   

\theorembodyfont{\normalfont}  
 
\newtheorem{remark}[proposition]{Remark}
\newtheorem{example}[proposition]{Example}

\theoremheaderfont{\normalfont\bfseries}  
\theorembodyfont{\itshape}  
\theoremsymbol{~}
\theoremseparator{.}  
\theoremstyle{nonumberplain}
\newtheorem{observation}{Observation}

\theoremstyle{nonumberplain}
\theorembodyfont{\normalfont}

\newcommand{\R}{\mathbb{R}}
\newcommand{\e}{\mathrm{e}}
\newcommand{\V}{\mathcal{V}}
\newcommand{\N}{\mathbb{N}}

\newcommand{\C}{\mathbb{C}}
\newcommand{\dd}{\mathrm{d}}

\newcommand{\tr}{\mathrm{tr}}

\newcommand{\vol}{\mathrm{vol}}

\newcommand{\id}{\mathrm{id}}

\renewcommand{\hat}{\widehat}
\newcommand{\<}{\left\langle}
\renewcommand{\>}{\right\rangle}

\title{Heat Kernel Asymptotics, Path Integrals and Infinite-Dimensional Determinants}
\author{ Matthias Ludewig}

\begin{document}

\maketitle

\begin{center}
  Max-Planck Institute for Mathematics\\
 Vivatgasse 7 / 53119 Bonn \\ \medskip
 matthias$\_$ludewig@gmx.de
\end{center}

\begin{abstract}
We compare the short-time expansion of the heat kernel on a Riemannian {mani-fold} with the formal stationary phase expansion of its representing path integral and prove that these asymptotic expansions coincide at lowest order. Besides shedding light on the formal properties of quantum mechanical path integrals, this shows that the lowest order term of the heat kernel expansion is given by the Fredholm determinant of the Hessian of the energy functional on the space of finite energy paths. We also relate this to the zeta determinant of the Jacobi operator, considering both the near-diagonal asymptotics as well as the behavior at the cut locus.
  \end{abstract}


\section{Introduction and Main Results}

Feynman's approach to quantum mechanics is based on path integrals, i.e.\  integrals over infinite-dimensional spaces of paths \cite{FeynmanHibbs}. This concept of functional integration, i.e.\ integration over spaces of fields, turned out to be immensely important not only in quantum mechanics but also in quantum field theory and many other areas of physics and mathematics \cite{Kleinert}. Mathematically however, these functional integrals are problematic to deal with at best, and this paper is part of the quest of gaining a rigorous understanding of this concept in the basic case of quantum mechanical path integrals. 

A guiding question is the following.
\begin{quote}
 {\textbf{Q:}} {\em In the cases where a path integral can be rigorously defined, do the properties of the well-defined object agree with those properties derived by formal manipulations of the path integral (i.e.\ the formal expression)?}
 \end{quote}
 In this paper, we will answer this question positively in a special instance: We will show that indeed, the short-time asymptotic expansion of the heat kernel on a compact Riemannian manifold agrees with the short-time asymptotic expansion derived by a (formal) stationary phase approximation from its path integral description. 

\medskip

{\textbf{The setup}. A basic path integral in quantum mechanics has the form
\begin{equation} \label{PathIntegral}
  \frac{1}{Z} \int_{x\rightsquigarrow y} e^{-S(\gamma)} F(\gamma) \mathcal{D} \gamma,
\end{equation}
where the integral is taken over the space of all continuous paths (parametrized by, say, the interval $[0, t]$) that travel between points $x$ and $y$ in space, $Z$ is a suitable normalization constant and 
\begin{equation} \label{ActionFunctional}
  S(\gamma) = \frac{1}{2} \int_0^t |\dot{\gamma}(t)|^2 \dd t
\end{equation}
is the standard action functional (traditionally called energy functional in differential geometry). 

There is a heuristic argument to answer the question which mathematical object should be represented by the path integral \eqref{PathIntegral}, at least in the case $F\equiv1$: If $x$ and $y$ are points in some Riemannian manifold, then the value of the path integral should coincide with the heat kernel of the Laplace-Beltrami operator $p_t(x, y)$ (for a heuristic explanation of this, see e.g.\ \cite[Section~1.1]{anderssondriver}). In formulas, 
\begin{equation} \label{PathIntegral2}
p_t(x, y) \stackrel{\text{formally}}{=} \frac{1}{Z} \int_{x\rightsquigarrow y} e^{-S(\gamma)/2} \mathcal{D} \gamma,
\end{equation}
where the division by $2$ in the exponent is conventional (the path integral without this factor would represent the heat kernel of the operator $\Delta/2$). More generally, taking $F(\gamma) = [\gamma\|_0^t]^{-1}$ in \eqref{PathIntegral}, the (inverse of the) parallel transport with respect to some connection on a vector bundle $\V$ over $M$, we will get the heat kernel of the associated connection Laplacian $L = \nabla^* \nabla$, acting on sections of $\V$.

The measure $\mathcal{D}\gamma$ in \eqref{PathIntegral} is supposed to denote integration with respect to the Riemannian volume measure corresponding to some Riemannian structure on the space of paths. The mathematical problem here is, however, that such a measure does not exist due to the infinite-dimensionality of the path space. Instead, \eqref{PathIntegral2} can be made rigorous by replacing the (infinite-dimensional) space of all continuous paths by (finite-dimensional) spaces of piece-wise geodesic paths subordinate to some subdivision of time, performing integration over this finite-dimensional manifold and then letting the mesh of the subdivision tend to zero. This has been extensively treated in various settings, see e.g.\ \cite{anderssondriver}, \cite{baerpfaeffle}, \cite{baerrenormalization}, \cite{LimPathintegrals}, \cite{Laetsch}, \cite{LiPathIntegrals}, \cite{ludewigThesis} or \cite{ludewigBoundary}. Another standard way making \eqref{PathIntegral} rigorous is using the Wiener measure on paths; this, however, somewhat obscures the role played by the action functional $S$ in the story.

\medskip

One of the most important (formal) features of a path integral investigated in physics is its stationary phase expansion, also called saddle point approximation (after a substitution, the short-time asymptotic expansion of the path integral \eqref{PathIntegral} corresponds to the semi-classical limit). Our first observation concerning our explicit path integral \eqref{PathIntegral} is the following.

\begin{observation}
There is a formal procedure to associate an asymptotic expansion as $t \rightarrow 0$ to the path integral \eqref{PathIntegral}, the coefficients of which are well-defined numbers.
\end{observation}

This is done by applying the formula for the coefficients of a stationary phase expansion (which is a theorem in finite-dimensions) to our path integral in a formal sense, which turns out to give well-defined coefficients. This will be explained precisely below.

Now, having discussed that the path integral  (at least formally) represents the heat kernel, a well-defined mathematical object, we can now ask the following special case of the motivating question formulated at the beginning of this article.
\begin{quote}{\textbf{Q'}:} {\em Does the formal stationary phase expansion in the limit $t \rightarrow 0$ of the right hand side of \eqref{PathIntegral2} agree with the short-time asymptotic expansion of the heat kernel?}\end{quote}
Below we explain how to associate an asymptotic expansion to the right hand side of \eqref{PathIntegral2}; then we describe our main results precisely, which gives a positive answer to the above question, to first order in $t$.
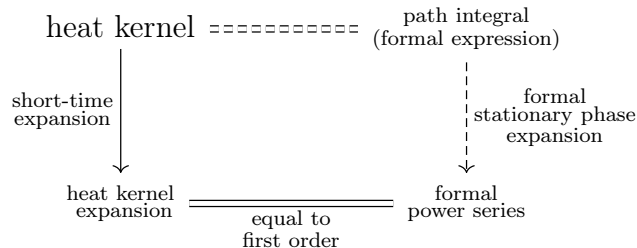
\begin{equation*}
\begin{tikzcd}
\text{heat kernel}\ar[dd, "\substack{\text{short-time}\\ \text{expansion}}"'] \ar[rr, dashed, equal] & &\substack{\text{path integral}  \\ \text{(formal expression)}} \ar[dd, "\substack{\text{formal} \\ \text{stationary phase} \\ \text{expansion}}", dashed]\\
& &\\
\substack{\text{heat kernel}\\ \text{ expansion}} \ar[rr, equal, "\substack{\text{equal to} \\ \text{first order}}"'] & & \substack{\text{formal} \\ \text{power series}}
\end{tikzcd}
\end{equation*}
The diagram illustrates our main statement: Even if we may not be able to define the path integral itself\footnote{In our situation, we can use finite-dimensional approximation to give the path integral formula some sense (as explained above), but this will not really define the path integral as an actual integral over path space.} (the upper right corner of the diagram), its formal stationary phase expansion can be defined rigorously, and it coincides with that of the heat kernel to first order.

\medskip

\medskip

{\textbf{The formal stationary phase expansion of the path integral.}} We now use a formal discussion to associate an asymptotic expansion to the right-hand side of \eqref{PathIntegral2}, the coefficients of which will turn out to be well-defined. First we need to eliminate the time-dependence in the integration domain: Making a suitable substitution, the expression \eqref{PathIntegral} can be brought into the form
\begin{equation} \label{PathIntegral3}
  I(t) = \frac{1}{Z} \int_{x\rightsquigarrow y} e^{-S(\gamma)/2t} F(\gamma) \mathcal{D} \gamma,
\end{equation}
where we integrate over some space $P_{xy}$ of paths travelling from $x$ to $y$ parametrized by $[0, 1]$ (which space of paths to take precisely will be discussed below). Now, from this formula, it is apparent that as $t$ goes to zero, the contributions of paths $\gamma$ with high action $S(\gamma)$ become very small so that the integral reduces in the limit to an integral over the minima of $S(\gamma)$, which are the minimal geodesics between $x$ and $y$. More precisely, if we pretend that $P_{xy}$ is a finite-dimensional Riemannian manifold, $\mathcal{D}\gamma$ denotes the corresponding Riemannian volume measure and $Z = Z(t) = (4 \pi t)^{\dim(P_{xy})/2}$, then we  obtain\footnote{Here and throughout, the tilde indicates that the quotient of the two sides converges to one as $t \rightarrow 0$.}
\begin{equation} \label{StationaryPhase1}
  I(t) ~\sim~ \frac{ e^{-S(\gamma_{xy})/2t} F(\gamma_{xy})}{\det(\nabla^2 S|_{\gamma_{xy}})^{1/2}}
\end{equation}
in the case that there is a unique minimizing geodesic $\gamma_{xy}$ between $x$ and $y$. Here we take the determinant of the Hessian of $S$ at $\gamma_{xy}$ with respect to the Riemannian metric on $P_{xy}$. More generally, if the set  $\Gamma_{xy}^{\min}$ of minimizing geodesics is a $k$-dimensional submanifold of the space of all paths, in the sense that the Hessian of $S$ is non-degenerate when restricted to the normal bundle of $\Gamma_{xy}^{\min}$, we obtain the formula
\begin{equation}\label{StationaryPhase2}
  I(t) ~\sim~ (4 \pi t)^{-k/2} \int_{\Gamma_{xy}^{\min}} \frac{e^{-S(\gamma)/2t} F(\gamma)}{{\det\bigl(\nabla^2 S|_{N_\gamma \Gamma_{xy}^{\min}}\bigr)^{1/2}} } \dd \gamma,
\end{equation} 
where this time, we take the determinant of the Hessian $\nabla^2 S$ restricted to the normal bundle to $\Gamma_{xy}^{\min}$ (which is non-zero due to the non-degeneracy assumption). This is a general version of  the stationary phase expansion or saddle point method, in its real version also called Laplace expansion. For a derivation for finite-dimensional manifolds, see \cite[Appendix~A]{LudewigStrongAsymptotics}. We remark that both in \eqref{StationaryPhase1} and \eqref{StationaryPhase2}, we need certain non-degeneracy conditions on the Hessian of $S$. 

Now of course, $P_{xy}$ will not be a finite-dimensional manifold, but we will see that with a suitable choice of path space and metric, the right-hand sides of \eqref{StationaryPhase1} and \eqref{StationaryPhase2} aquire well-defined meanings. Maybe the most natural choice is to take $P_{xy} := H_{xy}(M) \subset H^1([0, 1], M)$, the space of finite-energy paths. This is an infinite-dimensional Hilbert manifold, and the metric
\begin{equation} \label{H1Metric}
 (X, Y)_{H^1} = \int_0^1 \langle \nabla_s X(s), \nabla_s Y(s) \rangle \dd s
\end{equation}
for vector fields $X, Y$ along paths $\gamma$ turns it into an infinite-dimensional Riemannian manifold (here $\nabla_t$ denotes covariant differentiation along $\gamma$). It turns out that if the Hessian $\nabla^2 S$ of the action functional \eqref{ActionFunctional} is turned into an operator using this metric, the result will be an operator of determinant class, making the expression on the right hand side of \eqref{StationaryPhase1} well-defined. Moreover, the metric \eqref{H1Metric} induces a Riemannian metric to the submanifold $\Gamma_{xy}^{\min}$ which makes \eqref{StationaryPhase2} well-defined as well. We will refer to this choice as {\em the $H^1$-picture}.

Another natural choice is to let $P_{xy}$ be the space of smooth paths $\gamma$ with $\gamma(0) = x$, $\gamma(1) = y$, which is an infinite-dimensional manifold modelled on nuclear Fréchet spaces. Here the $L^2$-metric
\begin{equation} \label{L2Metric}
 (X, Y)_{L^2} = \int_0^1 \langle X(s), Y(s) \rangle \dd s
\end{equation}
seems to be a natural choice. In this case, turning $\nabla^2 S$ into an operator using this metric yields an unbounded elliptic operator, which possesses a well-defined zeta-regularized determinant. This gives rigorous interpretations to the right-hand sides of \eqref{StationaryPhase1} and \eqref{StationaryPhase2} in this case, which we will refer to as {\em the $L^2$-picture}.

\medskip

{\textbf{Main results.}} Let $M$ be a Riemannian manifold of dimension $n$. Our results are conveniently described by using the {\em Euclidean heat kernel}
\begin{equation} \label{EuclideanHeatKernel}
  \e_t(x, y) = (4 \pi t)^{-n/2} \exp\left( \frac{1}{4t} d(x, y)^2\right),
\end{equation}
the name of which stems from the fact that this is the heat kernel in Euclidean space.
Moreover, by $H_{xy}(M)$ we denote the space of absolutely continuous paths $\gamma$ with $\gamma(0) = x$, $\gamma(1) = y$ and such that the velocity field satisfies $\dot{\gamma} \in L^2([0, 1], \gamma^*TM)$, i.e.\ the paths in $H_{xy}(M)$ have finite energy \eqref{ActionFunctional}. The $H^1$-metric \eqref{H1Metric} turns $H_{xy}(M)$ into an infinite-dimensional Riemannian manifold.

Finally,  by $\Gamma_{xy}^{\min}$, we denote the set of length minimizing geodesics in $H_{xy}(M)$. We will always assume that $\Gamma_{xy}^{\min}$ is a submanifold of dimension $k$ and that it is {\em non-degenerate} in the sense that for each $\gamma \in \Gamma_{xy}^{\min}$ the restriction $\nabla^2 S|_{N_\gamma \Gamma_{xy}^{\min}}$ of the Hessian of the action functional restricted to the normal space $N_\gamma \Gamma_{xy}^{\min}$ of $\Gamma_{xy}^{\min}$ in $H_{xy}(M)$ is non-degenerate. Geometrically, this is equivalent to saying that for any Jacobi field $X$ along a geodesic $\gamma \in \Gamma_{xy}^{\min}$ that vanishes at both endpoints, there exists a geodesic variation $\gamma_s$ with fixed end points in direction $X$. For example, this assumption is always satisfied when $x$ and $y$ are closer than the injectivity radius so that $\Gamma_{xy}^{\min}$ only consists of the unique minimizing geodesic between the two points, and it also holds in the case that $x$ and $y$ are antipodal points of the round sphere $S^n$, in which case $\Gamma_{xy}^{\min}$ is $(n-1)$-dimensional.

\begin{theorem}[$H^1$-Picture] \label{ThmH1Picture}
Let $L$ be a formally self-adjoint Laplace type operator \mbox{acting} on sections of a metric vector bundle $\V$ over a compact Riemannian manifold $M$ of dimension $n$ and let $p_t^L(x, y)$ be the corresponding heat kernel. Suppose that for $x, y \in M$, the set of minimal geodesics $\Gamma_{xy}^{\min}$ is a $k$-dimensional non-degenerate submanifold of the path space $H_{xy}(M)$. Then 
\begin{equation} \label{H1Formula}
 \int_{\Gamma_{xy}^{\min}} \frac{[\gamma\|_0^1]^{-1}}{\det\bigl(\nabla^2 S|_{N_\gamma \Gamma_{xy}^{\min}}\bigr)^{1/2}} \dd^{H^1} \gamma =  \lim_{t \rightarrow 0} \, (4 \pi t)^{k/2} \frac{p_t^L(x, y)}{\e_t(x, y)},
\end{equation}
where $[\gamma\|_0^1]$ denotes the parallel transport along the path $\gamma$.
In particular, the bilinear form $\nabla^2 S|_{N_\gamma \Gamma_{xy}^{\min}}$ possesses a well-defined Fredholm determinant on the Hilbert space $N_\gamma \Gamma_{xy}^{\min}$ with respect to the $H^1$-metric \eqref{H1Metric}.
\end{theorem}

Notice that if $\gamma \in \Gamma_{xy}^{\min}$, then $S(\gamma) = d(x, y)^2/2$. Therefore, comparing with \eqref{StationaryPhase2}, this theorem states precisely that the short-time expansion of the heat kernel coincides with the formal stationary phase expansion of its representing path integral\footnote{Looking more closely, there is a discrepancy of a factor $(4 \pi t)^{-n/2}$. The reason is that the finite-dimensional approximation suggests that in the case of two fixed endpoints, $Z$ in \eqref{PathIntegral3} should really be given by $Z(t) = (4 \pi t)^{\dim(P_{xy})/2+n/2}$ in order to represent that heat kernel.} at first order. Hence in the $H^1$-picture, we obtain a positive answer to the motivating question in the introduction.

To prove Thm.~\ref{ThmH1Picture} in this paper, we use {\em finite-dimensional approximation} to relate the heat kernel to the path integral \eqref{PathIntegral}: Taking a partition $\tau = \{0 = \tau_0 < \tau_1 < \dots < \tau_N = 1\}$ of the time interval, we replace the infinite-dimensional manifold $H_{xy}(M)$ by the finite-dimensional submanifold $H_{xy;\tau}(M)$ of piecewise geodesics subordinated to the partition $\tau$. These finite-dimensional path integrals converge to the heat kernel as the mesh of the partition gets finer, as previous results of the author show. On the other hand, these finite-dimensional integrals can then be evaluated using Laplace's method, and careful error estimates show that one can exchange taking the limit $t \rightarrow 0$ and the limit $|\tau| \rightarrow 0$. 

\medskip

We now turn to a result involving the $L^2$-picture. A standard fact of Riemannian geometry is that the Hessian of the action functional is given by
\begin{equation}
  \nabla^2 S|_\gamma [X, Y] = \bigl(X, (-\nabla_s^2 + \mathcal{R}_\gamma)Y\bigr)_{L^2},
\end{equation}
where $(\mathcal{R}_\gamma Y)(s) = R(\dot{\gamma}(s), Y(s))\dot{\gamma}(s)$ denotes the Jacobi endomorphism, which motivates to replace the determinant of $\nabla^2 S|_{\gamma}$ in  \eqref{StationaryPhase2} by the zeta-regularized determinant of the unbounded operator $-\nabla_s^2 + \mathcal{R}_{\gamma}$. We then obtain the following result.

\begin{theorem} \label{ThmL2Picture}
Under the assumptions of Thm.~\ref{ThmH1Picture}, we have
\begin{equation} \label{L2Formula}
  \int_{\Gamma_{xy}^{\min}} [\gamma\|_0^1]^{-1} \frac{\det\nolimits_\zeta\bigl(-\nabla_s^2 \bigr)^{1/2}}{\det\nolimits_\zeta^\prime\bigl(-\nabla_s^2 + \mathcal{R}_\gamma \bigr)^{1/2}} \dd^{L^2} \gamma = \lim_{t \rightarrow 0} \, (4 \pi t)^{k/2} \frac{p_t^L(x, y)}{\e_t(x, y)}.
\end{equation}
Here we integrate with respect to the $L^2$-metric \eqref{L2Metric} on $\Gamma_{xy}^{\min}$ and $\det\nolimits_\zeta^\prime\bigl(-\nabla_s^2 + \mathcal{R}_\gamma \bigr)$ denotes the zeta determinant of the Jacobi operator with the zero modes excluded.
\end{theorem}

Notice here that the left-hand side needs to be normalized by the additional factor of $\det\nolimits_\zeta(-\nabla_s^2 )^{1/2}$ in order to obtain the correct result\footnote{We remark that this zeta-determinant can be explicitly calculated in general, namely $\det\nolimits_\zeta(-\nabla_s^2 ) = 2^n$, see Example~\ref{ExampleZetaLaplacian}.}. It is a ``well-known fact'' in physics that zeta determinants only give the value of path integrals ``up to an arbitrary multiplicative constant'', by which is usually meant that zeta determinants can only be used to calculate the {\em quotient} of two path integrals, which is then given by the quotient of the respective zeta determinants. This gives an explanation for the appearance of the quotient of two zeta determinants.

Again, with a view on \eqref{StationaryPhase2}, the theorem states that also in the $L^2$-picture, the short-time expansion of the heat kernel coincides with formal stationary phase expansion of the corresponding path integral. In this sense, the $H^1$-picture and the $L^2$-picture are equivalent.

\medskip

{\textbf{Further discussion.} Another aspect of this story is that Thm.~\ref{ThmH1Picture} or Thm.~\ref{ThmL2Picture} give interesting relations between geometric quantities on the path spaces of a Riemannian manifold and geometric quantities down on $M$. One such example is the Jacobian of the Riemannian exponential map,
\begin{equation} \label{JacobianOfExponentialMap}
  J(x, y) = \det \bigl(d \exp_x|_{\dot{\gamma}_{xy}(0)}\bigr),
\end{equation}
which is often called Van-Vleck-Morette-determinant in physics literature \cite[I.7]{TheMotionofPointParticlesinCurvedSpacetime}. $J(x, y)$ is a function on $M \bowtie M$, the set of points $(x, y) \in M \times M$ such that there exists a unique minimizing geodesic $\gamma_{xy}$ joining the two. In formula \eqref{JacobianOfExponentialMap}, we take the differential of the Riemannian exponential map at the point $\dot{\gamma}_{xy}(0)$ to obtain a linear map $d\exp_x: T_{\dot{\gamma}_{xy}(0)}T_x M \cong T_x M \longrightarrow T_y M$ and $J(x, y)$ is then the determinant of this linear map formed with help of the metric. It is well known (compare e.g.\ \cite[Section~2.5]{bgv}) that in the case that $(x, y) \in M \bowtie M$, we have
\begin{equation} \label{FirstOrderTermJacobian}
   \lim_{t \rightarrow 0} ~\frac{p_t^L(x, y)}{\e_t(x, y)} = J(x, y)^{-1/2} [\gamma_{xy}\|_0^1]^{-1}.
\end{equation}
Comparing this with the results of Thm.~\ref{ThmH1Picture} and Thm.~\ref{ThmL2Picture} gives the equalities
\begin{equation} \label{JacobianEquality}
J(x, y) = \det(\nabla^2S|_{\gamma_{xy}}) = \frac{\det\nolimits_\zeta(-\nabla_s^2 + \mathcal{R}_\gamma)}{ \det\nolimits_\zeta(-\nabla_s^2)}.
\end{equation}
While this may be useful to obtain information on $J(x, y)$ (for example, one directly sees that $J(x, y)$ is symmetric in $x$ and $y$, which is not obvious from its definition), it seems that the true power of this result lies in the reverse implications: Since $J(x, y)$ can be characterized as the solution to a certain ordinary differential equation (see \eqref{ExpODE}, \eqref{JacobiEquation} below), we obtain a method to calculate infinite-dimensional determinants by solving an ordinary differential equation. This result is known as the Gel'fand-Yaglom theorem and we will prove it as an application in Section~\ref{SectionGelfandYaglom}. A nice illustration is the following calculation in the constant curvature case.

\begin{example}[Constant Curvature Manifolds] \label{ExampleConstantCurvatureManifolds}
We calculate $\det(\nabla^2S|_\gamma)$ in the case that $\gamma \in \Gamma_{xy}^{\min} \subset H_{xy}(M)$ for a Riemannian manifold $M$ of constant sectional curvature $\kappa$. In this special case, the Jacobi eigenvalue equation along a geodesic $\gamma$ is (see e.g.\ \cite[p.~63]{ChavelEigenvalues})
\begin{equation*}
  \bigl(P+ \mathcal{R}_\gamma(s)\bigr)X(s) = -\nabla_s^2X(s) - \kappa |\dot{\gamma}(s)|^2 X(s) + \kappa \< X(s), \dot{\gamma}(s) \>\dot{\gamma}(s) = \lambda X(s).
\end{equation*}
Because $\gamma$ is a geodesic, the eigenspaces separate into spaces of vector fields that are either parallel to $\dot{\gamma}$ or orthogonal to $\dot{\gamma}$. Write $r:=|\dot{\gamma}(s)| = d(x, y)$ (which is independent of $s$ because $\gamma$ is a geodesic). Set $e_1(s) := \dot{\gamma}(s)/r$ and let $e_2(s), \dots, e_n(s)$ be a parallel orthonormal basis of the orthogonal complement of $\dot{\gamma}$ along $\gamma$. 

If we use the frame $e_1(s), \dots, e_n(s)$ to define the orthonormal basis $F_{ik}$ as in \eqref{ONBH1Vector}, then this is an orthonormal basis of eigenvectors of $P+\mathcal{R}_\gamma$ on the space $H^1_0([0, 1], \gamma^*TM)$: The $F_{1k}$ are eigenvectors to the eigenvalues $\lambda_k = \pi^2 k^2$ (so these have multiplicity one each), while the $F_{ik}$, $i=2, \dots, n$, are eigenvectors to the eigenvalues $\mu_k = \pi^2 k^2 - \kappa r^2$ (each of these has multiplicity $n-1$). The eigenvalues for the operator $\id + P^{-1}\mathcal{R}_\gamma$ are then
\begin{equation} \label{EigenvaluesOfHessianConstantCurvature}
  \tilde{\lambda_k} = \frac{\lambda_k}{\pi^2 k^2} = 1, ~~~~~~ \tilde{\mu}_k = \frac{\mu_k}{\pi^2 k^2} = 1 - \frac{\kappa r^2}{\pi^2k^2}.
\end{equation}
(If $\kappa>0$, this reflects that in order to have no zero eigenvalues, we need to have $r^2 \kappa < \pi^2$.) We obtain by \eqref{DeterminantAsProductOfEigenvalues} and \eqref{HessianEnergyOnH1}
\begin{equation*}
  \det\bigl(\nabla^2 S|_\gamma\bigr) = \det\bigl(\id + P^{-1}\mathcal{R}_\gamma\bigr) = \prod_{k=1}^\infty \left(1 - \frac{\kappa r^2}{\pi^2k^2}\right)^{n-1} = \left(\frac{\sin(\sqrt{\kappa}r)}{\sqrt{\kappa}r}\right)^{n-1}
\end{equation*}
by the product formula for the sine \cite[p.~220]{freitagbusam} (if $\kappa$ is negative, then $\sin$ becomes $\sinh$). These results coincide with the explicit formulas for the Jacobian of the exponential map $J(x, y)$ on manifolds with constant curvature  \cite[Example~5.1.2]{hsu}.
\end{example}

Furthermore, our results give formulas for the lowest order term of the heat kernel asymptotic expansion in the degenerate case, which seems to have been calculated in special cases only before (see e.g.\ \cite[Example~3.1]{molchanov} or \cite[Example~5.1.2]{hsu}). Moreover, using a degenerate Gel'fand-Yaglom theorem, we give a formula for the lowest order term in the heat kernel expansion by solving an ODE along geodesics, without mentioning infinite determinants (see Thm.~\ref{ThmDegenerateCaseJacobi} below).

We hope that by comparing the higher order terms in the asymptotic expansion of $p_t^L(x, y)$ with the complete formal asymptotic expansion of the path integral \eqref{PathIntegral}, one obtains further relations like this. This will be done in a subsequent paper.

\medskip

This paper is structured as follows. First, we review some theory on infinite-dimensional determinants in Hilbert spaces and we prove an approximation theorem on Fredholm determinants (Thm.~\ref{ThmDeterminantConvergence}) which seems to be new in this form, and which  is necessary for the proof of Thm.~\ref{ThmH1Picture}. In the next section, we show how to represent the heat kernel $p_t^L$ by approximating the path integral \eqref{PathIntegral} with finite-dimensional path integrals. Afterwards, we combine these techniques to give a proof of Thm.~\ref{ThmH1Picture}. In Section~\ref{SectionZetaDeterminants}, we give a brief summary of the theory of zeta determinants and prove Thm~\ref{ThmL2Picture}. Finally, we give an application of our results by proving the Gel'fand-Yaglom theorem, as mentioned above.

\medskip

\textbf{Acknowledgements}. I would like to thank Christian Bär, Rafe Mazzeo, Franziska Beitz and Florian Hanisch for helpful discussions, as well as the reviewer for helpful comments. Furthermore, I am indebted to Potsdam Graduate School, The Fulbright Program, SFB 647 ``Raum-Zeit-Materie'' and the Max-Planck-Institute for Mathematics Bonn for financial support.

\section{Fredholm Determinants and the Hessian of the \mbox{Action} Functional} \label{SectionSobolevSpacesAlongPaths}

For $a, b \in \R$, $a < b$, write $I := [a, b]$. Let $M$ be a Riemannian manifold of dimension $n$. For a smooth path $\gamma$ in $M$ parametrized by $I$, consider the operator $P := -\nabla_s^2$ on $L^2(I, \gamma^*TM)$ with Dirichlet boundary conditions. It is straightforward to show that $P$ is essentially self-adjoint on the domain $C^\infty_c(I, \gamma^*TM)$ (the space of compactly supported sections of $\gamma^*TM$) and self-adjoint on the Sobolev space $H^2_0(I, \gamma^*TM)$. Its eigenvalues can be explicitly computed: For a parallel orthonormal frame $e_1(s), \dots, e_n(s)$ of $TM$ along $\gamma$, the sections $E_{ik}$, $i=1, \dots, n$, $k=1, 2, \dots$, given by
\begin{equation} \label{ONBL2Vector}
  E_{ik}(a+s) := \sqrt{\frac{2}{b-a}} \sin\left(\frac{\pi k s}{b-a}\right) e_i(a+s), ~~~~~~ 0 \leq s \leq b-a
\end{equation}
form an orthonormal basis of $L^2(I, \gamma^*TM)$ consisting of eigenvectors of $P$. Obviously, the corresponding eigenvalues to $E_{ik}$ are the numbers
\begin{equation} \label{EigenvaluesOfP}
  \lambda_k := \frac{\pi^2 k^2}{(b-a)^2},~~~~~~ k = 1, 2, \dots
\end{equation}
each eigenvalue having multiplicity $n$.

Since the operator $P$ is positive and self-adjoint, we can form the powers $P^m$ for $m \in \R$ and define the Sobolev spaces
\begin{equation*}
  H^m_0(I, \gamma^*TM) := P^{-m/2}L^2(I, \gamma^*TM) \subset H^m(I, \gamma^*TM)
\end{equation*}
with the Sobolev norm
\begin{equation} \label{HmNorm}
  \|X\|_{H^m} := \|P^{m/2}X\|_{L^2},
\end{equation}
which is non-degenerate because $P$ has a trivial kernel.
By definition, this norm turns the map $P^{m/2}: H^l_0(I, \gamma^*TM) \longrightarrow H^{l-m}_0(I, \gamma^*TM)$ into an isometry, for any $m, l \in \R$.

Notice that for smooth $X \in H^1_0(I, \gamma^*TM)$, we have
\begin{equation*}
  (P^{1/2}X, P^{1/2} X)_{L^2} = (PX, X)_{L^2} = -(\nabla_s^2 X, X)_{L^2} = (\nabla_s X, \nabla_s X)_{L^2} = \|X\|_{H^1}^2
\end{equation*}
so that for $m=1$, the norm defined in \eqref{HmNorm} coincides with the $H^1$ norm defined before in \eqref{H1Metric} and there is no ambiguity in the notation. In particular, in the case that $I = [0, 1]$, we have
\begin{equation*}
  H^1_0(I, \gamma^*TM) \cong T_\gamma H_{xy}(M),
\end{equation*}
where $x := \gamma(0)$, $y := \gamma(1)$.
Of course, orthonormal bases on the spaces $H^m_0(I, \gamma^*TM)$ can be obtained by rescaling the $L^2$ orthonormal basis \eqref{ONBL2Vector} appropriately. In particular, the basis 
\begin{equation} \label{ONBH1Vector}
  F_{ik}(a+s) := \frac{\sqrt{2(b-a)}}{\pi k} \sin\left(\frac{\pi k s}{b-a}\right) e_i(a+s), ~~~~~~~~ 0 \leq s \leq b-a,
\end{equation}
$i=1, \dots, n$, $k=1, 2, \dots$, is an orthonormal basis of $H^1_0(I, \gamma^*TM)$.

For later use, we need the following two lemmas.

\begin{lemma} \label{LemmaInclusionHS}
For any $m \in \R$, the inclusion of $H^{m+1}_0(I, \gamma^*TM)$ into $H^m_0(I, \gamma^*TM)$ is a Hilbert-Schmidt operator. Furthermore, the inclusion operator from $H^{m+2}_0(I, \gamma^*TM)$ into $H^m_0(I, \gamma^*TM)$ is nuclear, and $P^{-1}$ is trace-class when considered as a bounded operator on $H^m_0(I, \gamma^*TM)$.
\end{lemma}

\begin{proof}
Denote the inclusion operator from $H^{m+1}_0$ into $H^m_0$ by $J_m$. In the case $m =1$, we have using the orthonormal basis \eqref{ONBH1Vector} of $H^1_0(I, \gamma^*TM)$ that
\begin{equation*}
  \|J_0\|_{2}^2 = \sum_{i=1}^n \sum_{k=1}^\infty \|J_0 F_{ik}\|_{L^2}^2 = \sum_{i=1}^n \sum_{k=1}^\infty \|F_{ik}\|_{L^2}^2 = n \sum_{k=1}^\infty \frac{(b-a)^2}{\pi^2 k^2} = (b-a)^2\frac{n}{6},
\end{equation*}
where we used that $\sum_{k=1}^\infty 1/k^2 = \pi^2/6$ \cite{EulerBaslerProblem}.
For $m \neq 1$, we have $J_m = P^{-m/2}J_0 P^{m/2}$, so that $J_m$ is also Hilbert-Schmidt by the ideal property of Hilbert-Schmidt operators.

The inclusion of $H^{m+2}_0(I, \gamma^*TM)$ into $H^m_0(I, \gamma^*TM)$ is equal to $J_m J_{m+1}$ and the composition of two Hilbert-Schmidt operators is trace-class, so the second statement follows. Finally, we can write
 \begin{equation*}
   \bigl[ P^{-1}:H^m_0\rightarrow H^m_0\bigr] = J_m \,J_{m+1}\,\bigl[ P^{-1}:H^m_0\rightarrow H^{m+2}_0\bigr],
\end{equation*}
which finishes the proof, because nuclear operators form an ideal.
\end{proof}

\begin{lemma} \label{LemmaOperatorNormOfP}
For any $l, m \in \R$ with $l \leq m$, we have 
\begin{equation*}
  \|P^{(l-m)/2}X\|_{H^m} = \|X\|_{H^l} \leq \left( \frac{b-a}{\pi}\right)^{m-l} \|X\|_{H^m}.
\end{equation*}
\end{lemma}

\begin{proof}
Using the basis $E_{ik}$ from \eqref{ONBL2Vector} to the eigenvalues $\lambda_k$, decompose a given vector field $X \in H^m_0([a, b], \gamma^*TM)$ as $X = \sum_{i=1}^n\sum_{k=1}^\infty X_{ik} E_{ik}$. Then for any $l \leq m$, we have
\begin{equation*}
  \|X\|_{H^m}^2 = \sum_{i=1}^n\sum_{k=1}^\infty \lambda_k^{m} |X_{ik}|^2 \geq \lambda_1^{m-l} \sum_{i=1}^n\sum_{k=1}^\infty \lambda_k^{l} |X_{ik}|^2 = \left( \frac{\pi^2}{(b-a)^2}\right)^{m-l} \|X\|_{H^l}^2,
\end{equation*}
using the explicit value for $\lambda_1$ as in \eqref{EigenvaluesOfP}. This is the statement.
\end{proof}

Let us now discuss the determinant of the Hessian of the action. If $T$ is a bounded linear operator on a separable Hilbert space $\mathcal{H}$, then its determinant can be defined if it has the form $T = \id + W$ with a trace-class operator $W$. We will call such operators {\em determinant-class} and their {\em (Fredholm) determinant} can be defined by
\begin{equation} \label{DeterminantAsProductOfEigenvalues}
 \det(T) := \prod_{j=1}^\infty (1+\lambda_j), 
\end{equation}
where $\lambda_j$ are the eigenvalues of $W$, repeated with algebraic multiplicity. Because as a trace-class operator, $W$ is compact, its non-zero spectrum consists only of eigenvalues of finite algebraic multiplicity (see e.g.\ Thm.~7.1 in \cite{ConwayFunctionalAnalysis}) and the trace-class condition means just that $\sum_{j=1}^\infty |\lambda_j| < \infty$, which implies that the product converges absolutely. In particular, $\det(T) = 0$ if and only if $\lambda_j = -1$ for some $j$, i.e.\ $T$ has the eigenvalue zero. There are many other ways to define the determinant of $T$, see \cite{SimonDeterminants}. For us, the following equivalent way to calculate a determinant will be useful.

\begin{theorem} \label{ThmDeterminantConvergence}
  Let $\mathcal{H}$ be a separable Hilbert space and let $T := \id + W$ be a bounded operator on $T$ with $W$ trace-class. Let $V_1 \subseteq V_2 \subseteq \dots$ be a nested sequence of finite-dimensional subspaces such that their union is dense in $\mathcal{H}$. Then we have
  \begin{equation*}
    \det(T) = \lim_{k\rightarrow \infty} \det\bigl(T|_{V_k}\bigr).
  \end{equation*}
\end{theorem}

\begin{remark} \label{RemarkDeterminantConvergence}
In particular, if $e_1, e_2, \dots$ is an orthonormal basis of $\mathcal{H}$, then setting $V_k$ to be the span of $e_1, \dots, e_k$ yields that
\begin{equation*}
\det(T) = \lim_{k\rightarrow \infty} \det\Bigl(\<e_i, Te_j\>\Bigr)_{1 \leq i, j \leq k},
\end{equation*}
where the latter is an ordinary determinant of matrices.
\end{remark}

\begin{proof}
Let $\pi_k$ be the orthogonal projection on $V_k$ and set $W_k = \pi_k W\pi_k$. Because $\id + W_k$ has the block diagonal form
  \begin{equation*}
    \id + W_k = \begin{pmatrix}
      T|_{V_k} & 0 \\ 0 & \id
    \end{pmatrix}
  \end{equation*}
  with respect to the orthogonal splitting $\mathcal{H} = V_k \oplus V_k^\perp$,
  we have
  \begin{equation*}
    \det\bigl(T|_{V_k}\bigr) = \det\bigl(\id + W_k\bigr),
  \end{equation*}
  where the right hand side denotes the Fredholm determinant on $\mathcal{H}$.
  Let $n_k$ be the dimension of $V_k$ and let $e_1, e_2, \dots$ be an orthonormal basis of $\mathcal{H}$ such that $e_1, \dots, e_{n_k}$ is an orthonormal basis of $V_k$. Using this orthonormal basis, we have
  \begin{equation} \label{EqTraceConvergence}
    \tr\,W_k = \sum_{j=1}^\infty \< e_j, \pi_k W \pi_k e_j\> = \sum_{j=1}^{n_k} \< e_j, W e_j\> \longrightarrow \sum_{j=1}^\infty \< e_j, W e_j\> = \tr\,W.
  \end{equation}
  For the Hilbert-Schmidt norm, we find
 \begin{equation*}
   \|W_k - W\|_{2}^2 = \sum_{ij = 1}^\infty \< e_i,(\pi_k W\pi_k - W) e_j\>^2 = \sum_{\{i, j \,\mid\, i > n_k\,\,\text{or}\,\,j >n_k\}} \< e_i, W e_j\>^2,
\end{equation*}  
  which converges to zero since $W$ is Hilbert-Schmidt (this follows e.g.\ from the dominated convergence theorem). Thus $W_k \rightarrow W$ in the Hilbert-Schmidt norm.

  The $2$-regularized determinant of a determinant-class operator $\id + V$ is defined by 
  \begin{equation*}
    \det\nolimits_2(\id + V) = \det(\id + V)e^{-\tr\,V},
  \end{equation*}
  see Section~6 in \cite{SimonDeterminants}. Because $\det\nolimits_2$ is continuous in the topology induced by Hilbert-Schmidt norm (Thm.~6.5 in \cite{SimonDeterminants}) and because of \eqref{EqTraceConvergence}, we have 
  \begin{equation*}
    \lim_{k \rightarrow \infty} \det(\id + W_k) = \lim_{k \rightarrow \infty} \det\nolimits_2(\id + W_k)  e^{\tr\,W_k} = \det\nolimits_2(\id + W)e^{\tr\,W} = \det(\id+W).
  \end{equation*}
  This finishes the proof.
\end{proof}

For $s \in I$, define the {\em Jacobi endomorphism} by
\begin{equation} \label{JacobiEndomorphism}
  \mathcal{R}_\gamma(s) v := R\bigl(\dot{\gamma}(s), v\bigr) \dot{\gamma}(s), ~~~~~~~~ v \in T_{\gamma(s)}M,
\end{equation}
where $R$ is the Riemann curvature tensor of $M$. Because of the symmetries of $R$, $\mathcal{R}_\gamma$ is a symmetric endomorphism field of the bundle $\gamma^*TM$ over $I$. Since $\mathcal{R}_\gamma$ is smooth and uniformly bounded on $I$, we can form the operator $P+\mathcal{R}_\gamma$, which is then self-adjoint on the same domain as $P$, and possesses the same mapping properties as $P$.

From now on, suppose that $\gamma$ is a geodesic. Then the Hessian $\nabla^2 S|_\gamma$ is given by
\begin{equation} \label{HessianEnergyOnL2}
  \nabla^2S|_{\gamma}[X, Y] = (\nabla_s X, \nabla_s Y)_{L^2} + (X, \mathcal{R}_\gamma Y)_{L^2} = \bigl(X, (P+\mathcal{R}_\gamma)Y\bigr)_{L^2}
\end{equation}
for $X, Y \in H^1_0(I, \gamma^*TM)$, see e.g.\ Thm.~13.1 in \cite{MilnorMorseTheory}.
Hence on $H^1_0(I, \gamma^*TM) \subset L^2(I, \gamma^*TM) $, the Hessian is given by the operator $P+\mathcal{R}_\gamma$ with respect to the $L^2$ metric. Of course, this operator is far from being determinant-class, since it is even unbounded. But by \eqref{HessianEnergyOnL2}, we also have
\begin{equation}\label{HessianEnergyOnH1}
  \nabla^2S|_{\gamma}[X, Y] =  (X,  Y)_{H^1} + (P^{-1}\mathcal{R}_\gamma X, Y)_{H^1} = \bigl(X, P^{-1}(P+\mathcal{R}_\gamma) Y\bigr)_{H^1},
\end{equation}
so on the space $H^1_0(I, \gamma^*TM)$, the bilinear form $\nabla^2S|_\gamma$ is given by the operator $P^{-1}(P+\mathcal{R}_\gamma) = \id + P^{-1}\mathcal{R}_\gamma$. Now, indeed, $P^{-1}\mathcal{R}_\gamma$ is trace-class on $H^1_0(I, \gamma^*TM)$, by Lemma~\ref{LemmaInclusionHS}. In fact, we can even calculate its value in terms of a curvature integral, as the following proposition shows.

\begin{proposition}[The Hessian of the Energy]\label{PropDetHessianExists}
 Let $\gamma \in H_{xy}(M)$ be a geodesic and consider $\nabla^2S|_\gamma$ as an operator  on $T_\gamma H_{xy}(M)$, by dualizing with the $H^1$ metric. Then $\nabla^2S|_\gamma - \id$ is trace-class with
 \begin{equation*}
   \mathrm{Tr}\bigl(\nabla^2S|_\gamma - \id\bigr) = - \int_0^1 \mathrm{ric}\bigl(\dot{\gamma}(s), \dot{\gamma}(s)\bigr) \,s(1-s)\,\dd s,
 \end{equation*}
 where $\mathrm{ric}$ denotes the Ricci tensor of $M$.
\end{proposition}

\begin{remark}
This implies that $\nabla^2S|_\gamma$ is determinant-class as a bilinear form on the Hilbert space $H^1_0([0, 1], \gamma^*TM)$ with respect to the $H^1$ metric. Furthermore, it is easy to see from the above considerations that $\nabla^2S|_\gamma$ is determinant-class on $H^m_0([0,1], \gamma^*TM)$ if and {  only if} $m = 1$.
\end{remark}

\begin{proof}
By \eqref{HessianEnergyOnH1}, we have using the orthonormal basis $F_{ik}$ from \eqref{ONBH1Vector} that
\begin{align*}
 \mathrm{Tr}\bigl(\nabla^2S|_\gamma - \id\bigr) &= \sum_{i=1}^n\sum_{k=1}^\infty \bigl(P^{-1}\mathcal{R}_\gamma F_{ik}, F_{ik}\bigr)_{H^1} = \sum_{i=1}^n\sum_{k=1}^\infty \bigl(\mathcal{R}_\gamma F_{ik}, F_{ik}\bigr)_{L^2}
\\ &= \int_0^1 \underbrace{\left(\sum_{j=1}^n \< R(\dot{\gamma}(s), e_i(s))\dot{\gamma}(s), e_i(s)\>\right)}_{=-\mathrm{ric}\bigl(\dot{\gamma}(s),\dot{\gamma}(s)\bigr)} \left(\sum_{k=1}^{\infty} \frac{2}{\pi^2 k^2}\sin\left(\pi s k\right)^2 \right)\dd s.
\end{align*}
Now because of $2\sin(z)^2 = 1 - \cos(2z)$, we have
\begin{equation*}
\sum_{k=1}^{\infty} \frac{2}{\pi^2 k^2}\sin\left(\pi s k\right)^2 
= \frac{1}{\pi^2}\sum_{k=1}^\infty \frac{1}{k^2} - \sum_{k=1}^{\infty} \frac{1}{\pi^2 k^2}\cos\left(2\pi s k\right)
= s(1-s),
\end{equation*}
where we used the Fourier transform identity of the second Bernoulli polynomial \cite{BernoulliFourierSeries},
\begin{equation*}
\sum_{k=1}^{\infty} \frac{1}{\pi^2 k^2}\cos\left(2\pi k s\right) = s^2 - s + \frac{1}{6}.
\end{equation*}
\end{proof}



\section{The Heat Kernel as a Path Integral} \label{SectionPathIntegrals}

Throughout, let $M$ be a compact Riemannian manifold of dimension $n$. We now discuss how to represent the heat kernel by finite-dimensional path integrals, which connects it to formula \eqref{PathIntegral}. For a partition $\tau = \{0=\tau_0 < \tau_1 < \dots < \tau_N = 1\}$ of the interval $[0, 1]$, write
\begin{equation} \label{SpacesOfPiecewiseGeodesics}
  H_{xy;\tau}(M) := \bigl\{ \gamma \in H_{xy}(M) \mid \gamma|_{[\tau_{j-1}, \tau_j]} ~\text{is a unique minimizing geodesic} \bigr\},
\end{equation}
for the space of paths $\gamma$ where each segment $\gamma|_{[\tau_{j-1}, \tau_j]}$ is a minimizing geodesic between its endpoints and the endpoints are not in each other's cut locus. One can show that $H_{xy;\tau}(M)$ is an $n(N-1)$-dimensional submanifold of $H_{xy}(M)$ and that the map
\begin{equation} \label{EvaluationMap}
  \mathrm{ev}_\tau: H_{xy;\tau}(M) \longrightarrow M^{N-1}, ~~~~~~~ \gamma \longmapsto \bigl(\gamma(\tau_1), \dots, \gamma(\tau_{N-1})\bigr)
\end{equation}
is an open embedding. The $H^1$-metric \eqref{H1Metric} turns $H_{xy;\tau}(M)$ into a finite-dimensional Riemannian manifold.

Let now $M$ be compact and let $L$ be a self-adjoint Laplace type operator, acting on sections of a metric vector bundle $\V$ over $M$. Any such operator can uniquely be written as $L = \nabla^*\nabla + V$, where $\nabla$ is a metric connection on $\V$ and $V$ is a symmetric endomorphism field. $L$ generates a strongly continuous semigroup of operators $e^{-tL}$, the heat semigroup. For each $t>0$ the operator $e^{-tL}$ has a smooth integral kernel $p_t^L(x, y)$, the {\em heat kernel}, which is a section of the bundle $\V \boxtimes \V^*$ over $M \times M$, the bundle with fibers $\V_x \otimes \V_y^* \cong \mathrm{Hom}(\V_y, \V_x)$ over points $(x, y) \in M \times M$. It has been shown by the author \cite{ludewigThesis} that in the case that $V=0$, one has
\begin{equation} \label{PathIntegralFiniteDimensional}
  p_t^L(x, y) = \lim_{|\tau|\rightarrow 0} \frac{1}{Z} \int_{H_{xy;\tau}(M)} e^{-S(\gamma)/2t}\, [\gamma\|_x^y]^{-1} \,\dd^{\Sigma\text{-}H^1} \gamma,
\end{equation}
where the limit goes over any sequence of partitions the mesh of which tends to zero, the integral over $H_{xy;\tau}(M)$ is taken with respect to a discretized version of the $H^1$ volume and $Z = Z(t) = (4 \pi t)^{\dim H_{xy;\tau}(M)/2 + n/2}$ (in the case that $V\neq 0$, a similar, slightly more complicated formula holds). 

\medskip

Unfortunately, the result \eqref{PathIntegralFiniteDimensional} cannot be used to perform the proof of Thm.~\ref{ThmH1Picture} envisioned in the introduction, because we have no control over the error of the path integral approximation: There is no reason why one should obtain the same result when taking the limit $|\tau| \rightarrow 0$ first and then $t \rightarrow 0$ as opposed to vice versa, especially if one divides by $\e_t(x, y)$ previously. 

To fix this, we add certain correction terms to the integrand of \eqref{PathIntegralFiniteDimensional}, which improve the time-uniformity of the approximation. In order to do this, we use that $p_t^L(x, y)$ has an asymptotic expansion near the diagonal of the form
\begin{equation*}
  p_t^L(x, y) \sim \e_t(x, y) \sum_{j=0}^\infty t^j \frac{\Phi_j(x, y)}{j!},
\end{equation*}
where the $\Phi_j$ are certain smooth sections of $\V \boxtimes \V^*$ near the diagonal, determined by transport equations, compare e.g.\ \cite[Section~2]{bgv} or \cite[Thm.~1.1]{LudewigStrongAsymptotics}. It is well known that $\Phi_0(x, y) = [\gamma_{xy}\|_0^1]^{-1}J(x, y)^{-1/2}$, where $[\gamma_{xy}\|_0^1]$ denotes parallel transport along the unique shortest geodesic $\gamma_{xy}$ between $x$ and $y$ and $J(x, y)$ is the Jacobian of the exponential map defined in \eqref{JacobianOfExponentialMap} (see \cite[Section~2.5]{bgv}).
In particular, for $(x, y) \in M \bowtie M$, we have
\begin{equation} \label{PhiKnotLimit}
   \lim_{t\rightarrow 0} ~\frac{p_t^L(x, y)}{\e_t(x, y)} = \Phi_0(x, y) = [\gamma_{xy}\|_0^1]^{-1}J(x, y)^{-1/2}.
\end{equation}

By the results of \cite{LudewigStrongAsymptotics}, the heat kernel $p_t^L(x, y)$ can be approximated by the path integral
\begin{equation} \label{FiniteDimensionalPathIntegral}
  I(t, x, y) := (4 \pi t)^{-nN/2} \int_{H_{xy;\tau}(M)} e^{-S(\gamma)/2t}\, \Upsilon^{\tau}(\gamma)\, \dd^{H^1} \gamma,
\end{equation}
where $\Upsilon^{\tau}(\gamma)$ is a complicated geometric term given by
\begin{equation} \label{UpsilonTauNu}
\begin{aligned}
  \Upsilon^{\tau}(\gamma)= \bigl|\det\bigl(d \mathrm{ev}_\tau|_\gamma\bigr)\bigr| \prod_{j=1}^N \chi\bigl(d(\gamma(\tau_{j-1}), \gamma(\tau_{j}))\bigr) \frac{\Phi_0\bigl(\gamma(\tau_{j-1}), \gamma(\tau_{j})\bigr)}{(\Delta_j\tau)^{n/2}}.
\end{aligned}
\end{equation}
It involves the Jacobian determinant of the evaluation map \eqref{EvaluationMap}, the first heat kernel coefficient $\Phi_0$ and a smooth cutoff function $\chi: [0, \infty) \rightarrow \R$ with $\chi(r) = 1$ near zero and $\chi(r) = 0$ for $r\geq \mathrm{inj}(M)$, the injectivity radius of $M$.
That $p_t^L(x, y)$ can be approximated by the integral $I(t, x, y)$ precisely means the following.

\begin{proposition} \label{PropPathIntegralApproximation}
There exist constants $\delta, C>0$ such that
\begin{equation*}
  \bigl| p_t^L(x, y) - I(t, x, y)\bigr| \leq C t  p_t^\Delta(x, y)
\end{equation*}
for all partitions $\tau$ of the interval $[0, 1]$ with $|\tau| \leq \delta$.
\end{proposition}

This is a special case of Thm.~1.2  combined with Lemma~5.7 in \cite{LudewigStrongAsymptotics} (one can make the error of smaller order in $t$ by adding more heat kernel coefficients to the definition \eqref{UpsilonTauNu} of $\Upsilon^\tau$. Relevant for us is the following corollary.

\begin{corollary} \label{CorollaryLimitsEqual}
  Under the assumptions of Thm.~\ref{ThmH1Picture}, we have
  \begin{equation*}
     \lim_{t\rightarrow 0} \, (4 \pi t)^{k/2} \frac{p_t^L(x, y)}{\e_t(x, y)} = \int_{\Gamma_{xy}^{\min}} \frac{\Upsilon^{\tau}(\gamma)}{\det\nolimits_\tau\bigl(\nabla^2 S|_{N_\gamma \Gamma_{xy}^{\min}}\bigr)^{1/2}} \dd^{H^1} \gamma.
  \end{equation*}
  for any partition $\tau$ with $|\tau| \leq \delta$, where $\delta>0$ is as in Prop.~\ref{PropPathIntegralApproximation}. Here by $\det\nolimits_\tau(\nabla^2 S|_{N_\gamma \Gamma_{xy}^{\min}})$, we mean the determinant of $\nabla^2 S$ when restricted to the normal space of $\Gamma_{xy}^{\min}$ inside $H_{xy;\tau}(M)$. 
  \end{corollary}

\begin{proof}
  By Prop.~\ref{PropPathIntegralApproximation} and the calculations above, there exist constants $\delta>0$ and $C_1>0$ such that
  \begin{equation} \label{EstimateICE}
    \left| \frac{p_t^L(x, y)}{\e_t(x, y)} - \frac{I(t, x, y)}{\e_t(x, y)} \right| \leq C_1 \frac{p_t^\Delta(x, y)}{\e_t(x, y)} t
  \end{equation} 
  for all $x, y \in M$, all partitions $\tau$ with $|\tau| \leq \delta$ and each $0 < t \leq 1$. By Thm.~5.2 in \cite{LudewigStrongAsymptotics}, we have
  \begin{equation} \label{EstimateQuotient}
    \frac{p_t^\Delta(x, y)}{\e_t(x, y)} \leq C_2 t^{-k/2}
   \end{equation}
   for some constant $C_2>0$. Using this on \eqref{EstimateICE}, multiplying by $(4 \pi t)^{k/2}$ and taking the limit $t \rightarrow 0$, we get, using the definition of $I(t, x, y)$,
     \begin{equation*}
    \lim_{t\rightarrow 0} \, (4 \pi t)^{k/2} \frac{p_t^L(x, y)}{\e_t(x, y)} = \lim_{t\rightarrow 0} \,(4 \pi t)^{-n(N-1)/2+k/2} \fint_{H_{xy;\tau}(M)} e^{-[S(\gamma) - d(x, y)^2/2]/2t}\, \Upsilon^{\tau}(\gamma) \,\dd^{H^1} \gamma.
  \end{equation*}
The limit on the right hand side can now be evaluated using Laplace's method if $\Gamma_{xy}^{\min}$ is a $k$-dimensional non-degenerate submanifold of $H_{xy}(M)$. Namely, in this case $\Gamma_{xy}^{\min}$ is also a non-degenerate submanifold of $H_{xy;\tau}(M)$ for each partition $\tau$ fine enough. The function $\phi(\gamma):=S(\gamma) - d(x, y)^2/2$ is zero on the submanifold $\Gamma_{xy}^{\min}$ and positive everywhere else. Therefore, by Laplace's method (compare e.g.\ Appendix~A in \cite{LudewigStrongAsymptotics}), the integral therefore concentrates on $\Gamma_{xy}^{\min}$ in the limit $t \rightarrow 0$. The precise result is
\begin{equation*}
\lim_{t\rightarrow 0}\,(4 \pi t)^{-n(N-1)/2+k/2} \int_{H_{xy;\tau}(M)} \!\!\!\!\!e^{-\phi(\gamma)/2t}\, \Upsilon^{\tau}(\gamma) \,\dd^{H^1} \gamma = \int_{\Gamma_{xy}^{\min}} \frac{\Upsilon^{\tau}(\gamma)}{\det\nolimits_\tau\bigl(\nabla^2 \phi |_{N_\gamma \Gamma_{xy}^{\min}}\bigr)^{1/2}} \dd^{H^1} \gamma,
\end{equation*}
where the determinant of $\nabla^2\phi$ is taken over the (finite-dimensional) normal space of $\Gamma_{xy}^{\min}$ inside $H_{xy;\tau}(M)$. Clearly, $\nabla^2 \phi = \nabla^2 S$ so the result follows.
\end{proof}

\section{Heat Kernel Asymptotics as a Fredholm Determinant} \label{SectionHeatKernelFredholm}

In this section, we prove Thm.~\ref{ThmH1Picture}. Before starting the proof, let us shed some light on the assumptions of the theorem.

\begin{lemma}
Let $M$ be a complete manifold and suppose that for $x, y \in M$, the set $\Gamma_{xy}^{\min}$ of length minimizing geodesics is a submanifold of $H_{xy}(M)$. Then the following statements are equivalent.
\begin{enumerate}[(i)]
\item $\Gamma_{xy}^{\min}$ is non-degenerate in the sense that for each $\gamma \in \Gamma_{xy}^{\min}$, the Hessian of the action $\nabla^2 S$ is non-degenerate when restricted to the normal bundle $N_\gamma \Gamma_{xy}^{\min}$ of $\Gamma_{xy}^{\min}$ at $\gamma$.
\item For each $\gamma \in \Gamma_{xy}^{\min}$ and each Jacobi field $X$ along $\gamma$ with $X(0) = X(1) = 0$, there exists a geodesic variation $\gamma_\varepsilon$, $\varepsilon \in (-\epsilon_0, \epsilon_0)$ with $\frac{\partial}{\partial \varepsilon}|_{\varepsilon=0} \gamma_\varepsilon = X$ and $\gamma_\varepsilon(0) = x$, $\gamma_\varepsilon(1) = y$ for each $\varepsilon \in (-\epsilon_0, \epsilon_0)$.
\end{enumerate}
\end{lemma}

Notice that if $X \in T_\gamma \Gamma_{xy}^{\min}$, we always have $\nabla^2S[X, Y] = 0$ for all $Y \in T_\gamma H_{xy}(M)$, so that unless $\dim(\Gamma_{xy}^{\min})= 0$, $\nabla^2 S$ is always degenerate on $T_\gamma H_{xy}(M)$.

\begin{proof}
Let $\gamma \in \Gamma_{xy}^{\min}$ and suppose there exists $X \in N_\gamma \Gamma_{xy}^{\min}$ such that $\nabla^2S|_\gamma[X, Y]   = 0$ for all $Y \in T_\gamma H_{xy}(M)$. By \eqref{HessianEnergyOnL2}, $X$ is a weak solution of the equation $(-\nabla_s^2 + \mathcal{R}_\gamma) X = 0$, therefore smooth by elliptic regularity and hence a strong solution. That is, $X$ is a Jacobi field with $X(0) = X(1) = 0$. However, there exists no geodesic variation $\gamma_\varepsilon$ in $H_{xy}(M)$ with $\frac{\partial}{\partial \varepsilon}|_{\varepsilon=0} \gamma_\varepsilon = X$, because then we would have $\gamma_\varepsilon \in \Gamma_{xy}^{\min}$ for each $\varepsilon$ and $\frac{\partial}{\partial \varepsilon}|_{\varepsilon =0} \gamma_\varepsilon \in T\Gamma_{xy}^{\min}$, a contradiction to $X \in N_\gamma \Gamma_{xy}^{\min}$.

Conversely, let $X$ be a Jacobi field along $\gamma \in \Gamma_{xy}^{\min}$ with $X(0) = X(1) = 0$. If there exists no geodesic variation $\gamma_\varepsilon \in \Gamma_{xy}^{\min}$ with $\frac{\partial}{\partial \varepsilon}|_{\varepsilon=0} \gamma_\varepsilon = X$, then this means that $X \notin T_\gamma \Gamma_{xy}^{\min}$. Hence its normal component $\tilde{X}$ is not zero, and $\tilde{X}$ is a Jacobi field, because all elements of $T_\gamma \Gamma_{xy}^{\min}$ are Jacobi fields, and $\tilde{X}$ is the difference of such a tangent vector and the Jacobi field $X$. This implies that $\nabla^2 S|_\gamma[\tilde{X}, Y] = 0$ for all $Y \in T_\gamma H_{xy}(M)$, so that $\nabla^2 S|_\gamma$ is degenerate, even when restricted to the normal space $N_\gamma \Gamma_{xy}^{\min}$.
\end{proof}

\begin{remark}
Versions of Thm.~\ref{ThmH1Picture} can also be obtained in the case that $\Gamma_{xy}^{\min}$ is a {\em degenerate} submanifold of $H_{xy}(M)$, compare e.g.\ \cite[Section~3]{molchanov}. We restrict to the non-degenerate case for simplicity.
\end{remark}

\begin{remark}
It is easy to see that a version of Thm.~\ref{ThmH1Picture} also holds in the case that $\Gamma_{xy}^{\min}$ is a disjoint union of a $k$-dimensional submanifold $\Gamma_0$ and submanifolds $\Gamma_1, \dots, \Gamma_\ell$ of lower dimension, under the assumption that all of them are all non-degenerate. In that case, \eqref{H1Formula} still holds if one replaces the integral over $\Gamma_{xy}^{\min}$ by an integral over $\Gamma_0$.
\end{remark}

Let us now go on with the proof of Thm.~\ref{ThmH1Picture}. Remember that so far (Corollary~\ref{CorollaryLimitsEqual}), we have shown that
\begin{equation*}
     \lim_{|\tau|\rightarrow 0} \, (4 \pi t)^{k/2} \frac{p_t^L(x, y)}{\e_t(x, y)} = \int_{\Gamma_{xy}^{\min}} \frac{\Upsilon^{\tau}(\gamma)}{\det\nolimits_\tau\bigl(\nabla^2 S|_{N_\gamma \Gamma_{xy}^{\min}}\bigr)^{1/2}} \dd^{H^1} \gamma.
  \end{equation*}
  for any partition $\tau$ small enough, where $\Upsilon^\tau(\gamma)$ is given by \eqref{UpsilonTauNu} and $\det\nolimits_\tau(\nabla^2 S|_{N_\gamma \Gamma_{xy}^{\min}})$ is the finite-dimensional determinant of $\nabla^2 S$, restricted to the normal space of $\Gamma_{xy}^{\min}$ inside $H_{xy;\tau}(M)$. It remains to show that the integrand under the integral over $\Gamma_{xy}^{\min}$ can be replaced by the integrand of Thm.~\ref{ThmH1Picture}. To obtain Thm.~\ref{ThmH1Picture} from this, we need the following three technical lemmas, the (somewhat involved) proofs of which will be given at the end of this section.

  With a view on the formula  \eqref{UpsilonTauNu}, it should be clear the first two lemmas are needed to show that $\Upsilon^\tau(\gamma)$ converges to the inverse parallel transport as $t\rightarrow 0$, while we employ the third lemma to argue that the determinants of $\nabla^2 S$ on the finite-dimensional normal bundles converge to the infinite-dimensional determinant appearing in Thm.~\ref{ThmH1Picture}.
  
\begin{lemma} \label{LemmaPhiKnot}
There exists a constant $C>0$ such that for each $\gamma \in \Gamma_{xy}^{\min}$, we have
\begin{equation*}
  \left| \prod_{j=1}^N \Phi_0\bigl(\gamma(\tau_{j-1}), \gamma(\tau_j)\bigr) - [\gamma\|_0^1]^{-1}\right| \leq C |\tau|
\end{equation*}
for each partition $\tau$ of the interval $[0, 1]$ fine enough.
\end{lemma}

\begin{lemma} \label{LemmaDeterminantOfEvaluationMap}
Let $M$ be a compact Riemannian manifold and $x, y \in M$. Then for every $C>0$, there exist constants $\alpha>0$ and $N_0 \in \N$ such that the following holds: For any geodesic $\gamma \in \Gamma_{xy}^{\min}$ in $M$, we have
\begin{equation*}
  e^{- \alpha |\tau|^{3}} \leq \bigl|\det \bigl(d \mathrm{ev}_\tau|_\gamma \bigr)\bigr| \prod_{j=1}^N (\Delta_j\tau)^{-n/2} \leq  e^{\alpha |\tau|^{3}}
\end{equation*}
for any partition $\tau = \{ 0 = \tau_0 < \tau_1 < \dots < \tau_N = 1\}$ of the interval $[0, 1]$ such that $N \geq N_0$ and $|\tau| \leq C/N$.
\end{lemma}

\begin{lemma} \label{LemmaTangentSpacesDense}
  Let $S$ be a set of partitions of the interval $[0, 1]$ such that for any $\varepsilon>0$, there exists $\tau \in S$ with $|\tau|<\varepsilon$. Then for any $\gamma \in \Gamma_{xy}^{\min}$, the union of the spaces $T_\gamma H_{xy;\tau}(M)$, $\tau \in S$ is dense in $T_\gamma H_{xy}(M) = H^1_0([0, 1], \gamma^*TM)$.
\end{lemma}

Using these Lemmas, we can now prove our main result, the theorem on the $H^1$-picture.

\begin{proof}[of Thm.~\ref{ThmH1Picture}]
By Corollary~\ref{CorollaryLimitsEqual}, we are done if we show that
\begin{equation*}
\lim_{|\tau|\rightarrow 0} \int_{\Gamma_{xy}^{\min}} \frac{\Upsilon^{\tau}(\gamma)}{\det\nolimits_\tau\bigl(\nabla^2 S|_{N_\gamma \Gamma_{xy}^{\min}}\bigr)^{1/2}} \dd^{H^1} \gamma = \int_{\Gamma_{xy}^{\min}} \frac{[\gamma\|_0^1]^{-1}}{\det\bigl(\nabla^2 S|_{N_\gamma \Gamma_{xy}^{\min}}\bigr)^{1/2}}  \dd^{H^1} \gamma.
\end{equation*}
where we may take the limit over any suitable sequence $(\tau^{(k)})_{k \in \N}$ of partitions satisfying $|\tau^{(k)}| \rightarrow 0$ (notice that a particular result of the lemma is that the value of the integral on the left hand side integral does not depend on $\tau$).

By Lemma~\ref{LemmaPhiKnot} and Lemma~\ref{LemmaDeterminantOfEvaluationMap}, the function $\Upsilon^{\tau}(\gamma)$ is uniformly bounded on $\Gamma_{xy}^{\min}$ and we have for any $\gamma \in \Gamma_{xy}^{\min}$ that
\begin{equation*}
  \lim_{|\tau|\rightarrow 0} \Upsilon^{\tau}(\gamma) = [\gamma\|_0^1]^{-1}.
\end{equation*}
Here for a fixed $C>0$, the limit goes over any sequence $\tau^{(k)}$ of partitions with $|\tau^{(k)}|\rightarrow 0$ and $|\tau^{(k)}| \leq C/N$. 

By Lemma~\ref{LemmaTangentSpacesDense}, for any such sequence $(\tau^{(k)})_{k \in \N}$, the union of the  spaces $T_\gamma H_{xy;\tau^{(k)}}(M)$, $k \in \N$, is dense in $T_\gamma H_{xy}(M)$ for every $\gamma \in \Gamma_{xy}^{\min}$. Furthermore, also the union of the spaces $N_{\gamma}\Gamma_{xy}^{\min} \cap T_\gamma H_{xy;\tau^{(k)}}(M)$ is dense in $N_\gamma \Gamma_{xy}^{\min}$.
(For let $X \in N_\gamma \Gamma_{xy}^{\min}$. Then by  Lemma~\ref{LemmaTangentSpacesDense}, there exists a sequence $X_k \in T_\gamma H_{xy;\tau^{(k)}}(M)$ with $\|X-X_k\|_{H^1}\rightarrow 0$. But if $Y_k \in T_\gamma \Gamma_{xy}^{\min}$ is the part of $X_k$ tangent to $\Gamma_{xy}^{\min}$, we have
\begin{equation*}
  \|X-X_k\|_{H^1}^2 = \bigl\|X-(X_k - Y_k)\bigr\|_{H^1}^2 + \|Y_k\|_{H^k}^2,
\end{equation*}
so that $X_k - Y_k$ is an approximating sequence of $X$ in $N_{\gamma}\Gamma_{xy}^{\min} \cap T_\gamma H_{xy;\tau^{(k)}}(M)$.)
By Thm.~\ref{ThmDeterminantConvergence}, we therefore have
\begin{equation*}
  \lim_{k \rightarrow \infty }\det\nolimits_{\tau^{(k)}}\bigl(\nabla^2 S|_{N_{\gamma} \Gamma^{\min}_{xy}}\bigr) = \lim_{k \rightarrow \infty }\det\bigl(\nabla^2 S|_{N_{\gamma} \Gamma^{\min}_{xy}\cap T_\gamma H_{xy;\tau^{(k)}}(M)}\bigr)  = \det\nolimits\bigl(\nabla^2 S|_{N_{\gamma} \Gamma^{\min}_{xy}}\bigr) 
\end{equation*}
if we additionally choose the sequence $(\tau^{(k)})_{k \in \N}$ to be nested (since then the corresponding sequence of spaces $N_{\gamma} \Gamma^{\min}_{xy} \cap T_\gamma H_{xy;\tau^{(k)}}(M)$ is nested, too).

We obtain that if for a fixed $C>0$, we take the limit over some nested sequence of partitions $\tau^{(k)}$ with $|\tau^{(k)}|\rightarrow 0$
that additionally satisfies $|\tau^{(k)}| \leq C/N$, then the integrand from Corollary~\ref{CorollaryLimitsEqual} converges to the integrand from the theorem pointwise. 

To justify the exchange of integration and taking the limit, we give a uniform bound on $\det\nolimits_{\tau}\bigl(\nabla^2 S|_{N_{\gamma} \Gamma^{\min}_{xy}}\bigr)$, then the result follows from  Lebesgue's theorem of dominated convergence. Because of \eqref{HessianEnergyOnH1}, we have
\begin{equation*}
  \det\nolimits_{\tau}\bigl(\nabla^2 S|_{N_{\gamma} \Gamma^{\min}_{xy}}\bigr) := \det\bigl(\nabla^2 S|_{N_\gamma \Gamma^{\min}_{xy}\cap H_{xy;\tau}(M)}\bigr) = \det \bigl((\id + \pi_\tau P^{-1} \mathcal{R}_\gamma\pi_\tau)|_{N_\gamma \Gamma^{\min}_{xy}} \bigr),
\end{equation*}
where $\pi_\tau$ is the orthogonal projection of $H_{xy}(M)$ onto $H_{xy;\tau}(M)$.
By the standard estimate for Fredholm determinants (see \cite[Thm.~3.2]{SimonDeterminants})
\begin{equation} \label{StandardDeterminantEstimate}
  e^{-\|T\|_1} \leq \det(\id + T) \leq e^{\|T\|_1},
\end{equation}
which holds for all trace-class operators $T$, we have 
\begin{equation*}
  \det\nolimits_{\tau}\bigl(\nabla^2 S|_{N_{\gamma} \Gamma^{\min}_{xy}}\bigr)^{-1/2} \leq e^{\|\pi_\tau P^{-1}\mathcal{R}_\gamma\pi_\tau\|_1/2}.
\end{equation*}
But
\begin{equation*}
  \|\pi_\tau P^{-1}\mathcal{R}_\gamma\pi_\tau\|_{1} \leq \|\pi_\tau\|\|P^{-1} \mathcal{R}_\gamma\|_1 \|\pi_\tau\| \leq \|P^{-1}\|_1 \|\mathcal{R}_\gamma\|,
\end{equation*}
which is finite by Lemma~\ref{LemmaInclusionHS} and bounded uniformly over $\gamma \in \Gamma_{xy}^{\min}$ since $\Gamma_{xy}^{\min}$ is compact. This establishes the required bound and finishes the proof.
\end{proof}

Restricting to the case $(x, y) \in M \bowtie M$ gives the following corollary.

\begin{corollary}[The Jacobian of the Exponential Map] \label{CorollaryJacobian}
Let $M$ be a complete Riemannian manifold. Let  $(x, y) \in M \bowtie M$ and let $\gamma_{xy}$ be the unique minimizing geodesic connecting $x$ to $y$ in time one. Then we have 
  \begin{equation*}
  \det\bigl(\nabla^2 S|_{\gamma_{xy}}\bigr) = J(x, y),
\end{equation*}
where $J(x, y)$ is the Jacobian determinant of the exponential map, as defined in \eqref{JacobianOfExponentialMap}. Here, $H_{xy}(M)$ carries the $H^1$ metric \eqref{H1Metric}.
\end{corollary}

\begin{proof}
Of course, this is a local result, so in the case that $M$ is non-compact, we can take some patch of $M$ containing $\gamma_{xy}$ and embed it isometrically into some compact Riemannian manifold $M^\prime$ in such a way that $\gamma_{xy}$ is still a minimizing geodesic, without changing $J(x, y)$ or the determinant of the Hessian of the action. This shows that we may assume that $M$ is compact so that the above results apply.

Taking the heat kernel of the Laplace-Beltrami operator in Thm.\ \ref{ThmH1Picture} and restricting to the case $(x, y) \in M \bowtie M$ (which implies $\Gamma_{xy}^{\min} = \{\gamma_{xy}\}$ and $k = \dim(\Gamma_{xy}^{\min}) = 0$), we have
\begin{equation*}
   J(x, y)^{-1/2} = \Phi_0(x, y) =\lim_{t \rightarrow 0} \frac{p^\Delta_t(x, y)}{\e_t(x, y)} = \det\bigl(\nabla^2 S|_{\gamma_{xy}}\bigr)^{-1/2}.
 \end{equation*}
 by \eqref{FiniteDimensionalPathIntegral}.
\end{proof}

\begin{example}[The first Coefficient on Spheres]
On an $n$-dimensional sphere $S^n_R$ of radius $R = 1/\sqrt{\kappa}$, the determinant of the Hessian of the action, respectively the Jacobian of the exponential map, is given by 
\begin{equation} \label{JacobianConstantCurvature}
  J(x, y) = \left(\frac{\sin\bigl(\sqrt{\kappa}\,d(x, y)\bigr)}{\sqrt{\kappa}\, d(x, y)}\right)^{n-1},
\end{equation}
in the case that $x$ and $y$ are not antipodal points, see \cite[Example~5.1.2]{hsu}. From this, one can read off the heat kernel asymptotics of Thm.~\ref{ThmH1Picture} in this case (compare \eqref{PhiKnotLimit} and Corollary~\ref{CorollaryJacobian}). We now use Thm.\ \ref{ThmH1Picture} to calculate the first coefficient for the Laplace-Beltrami operator on $S^n_R$ in the case that $x$ and $y$ are antipodal points. 

Without loss of generality, we assume that $x = (R, 0, \dots, 0)$ and $y = (-R, 0, \dots, 0)$ are the north and south pole. In this case, the set $\Gamma_{xy}^{\min}$ is diffeomorphic to $S_R^{n-1}$, the $n-1$-dimensional sphere of radius $R$, via the diffeomorphism
\begin{equation*}
\begin{aligned}
  \rho: S^{n-1}_R \longrightarrow \Gamma^{\min}_{xy}~~~~~~~~~~
  \theta \longmapsto \Bigl[ s \mapsto \begin{pmatrix}
  R \cos(\pi s) \\\theta \sin(\pi s)  \end{pmatrix} \Bigr].
\end{aligned}
\end{equation*}
For $v \in T_\theta S^{n-1}_R$, we have
\begin{equation*}
  d \rho|_\theta v =: X_v = \Bigl[ s \mapsto \begin{pmatrix}
   0 \\ v \sin(\pi s)\end{pmatrix} \Bigr].
\end{equation*}
Since $v \in T_\theta S^{n-1}_R$, we have $\<v, \theta\>=0$, hence $\< \dot{X}_v(s), \rho(\theta)(s)\> = 0$ so that
\begin{equation*}
  \nabla_s X_v(s) = \dot{X}_v(s) - \kappa\< \dot{X}_v(s), \rho(\theta)(s)\>\rho(\theta)(s)= -\pi\begin{pmatrix}
   0 \\ v \sin(\pi s)\end{pmatrix},
\end{equation*}
by the explicit formula for the Levi-Civita connection on the round sphere.
Therefore, if $e_1, \dots, e_{n-1}$ is an orthonormal basis of $T_\theta S_R^{n-1}$, the Jacobian determinant of $\rho$ is given by
\begin{equation*}
\begin{aligned}
  \bigl|\det \bigl(d \rho|_\theta\bigr)\bigr| &= \det \Bigl(\bigl(X_{e_i}, X_{e_j}\bigr)_{H^1}\Bigr)^{1/2}_{1 \leq i, j \leq n-1}  = \det \left( \pi^2 \<e_i, e_j\> \int_0^1 \cos(\pi s)^2 \dd s \right)^{1/2}_{1 \leq i, j \leq n-1} \\
   &= \pi^{n-1} 2^{(1-n)/2},
\end{aligned}
\end{equation*}
which is constant. To calculate the determinant of the Hessian of the action, remember that the eigenvalues are given by \eqref{EigenvaluesOfHessianConstantCurvature}. In our case, $r = R\pi$ and $\kappa = 1/R^2$ so $\tilde{\mu}_1=0$, which has to be left out to calculate the Hessian on the normal space to $\Gamma_{xy}^{\min}$. We obtain
\begin{equation*}
 \det\bigl(\nabla^2 S|_{N_\gamma \Gamma_{xy}^{\mathrm{min}}}\bigr) = \prod_{k=2}^\infty \tilde{\mu}_k^{n-1} = \prod_{k=2}^\infty \left(1 - \frac{\kappa r^2}{\pi^2 k^2}\right)^{n-1} = \prod_{k=2}^\infty \left(1 - \frac{1}{k^2}\right)^{n-1} = 2^{1-n},
\end{equation*}
because the product ``telescopes'', that is
\begin{equation*}
  \prod_{k=2}^\infty \left(1 - \frac{1}{k^2}\right) = \lim_{N\rightarrow \infty}\left(\prod_{k=2}^N \frac{k-1}{k}\right)\left(\prod_{k=2}^N \frac{k+1}{k}\right) = \lim_{N\rightarrow \infty}\frac{1}{N}\frac{N+1}{2} = \frac{1}{2}.
\end{equation*}
Therefore, by Thm.\ \ref{ThmH1Picture}, we have
\begin{equation*}
\begin{aligned}
  \lim_{t\rightarrow 0}~ (4\pi t)^{(1-n)/2}\, \frac{p_t^L(x, y)}{\e_t(x, y)} &= \int_{\Gamma_{xy}^{\min}} 2^{(n-1)/2}\dd^{H^1} \gamma = 2^{(n-1)/2} \int_{S^{n-1}_R}\det\bigl(d\rho|_\theta\bigr)\dd \theta \\
  &= \pi^{n-1}R^{n-1} \vol(S^{n-1}) = 2 \frac{\pi^{3n/2 -1}R^{n-1}}{\Gamma\left(n/2\right)}.
\end{aligned}
\end{equation*}
This result can also be found in \cite[Example 5.3.3]{hsu}.
\end{example}

To finish this section, it is left to prove the Lemmas \ref{LemmaPhiKnot}, \ref{LemmaDeterminantOfEvaluationMap} and \ref{LemmaTangentSpacesDense}.

\begin{proof}[of Lemma~\ref{LemmaPhiKnot}]
 By \eqref{PhiKnotLimit}, we have 
\begin{equation} \label{PhiKnotAndJ}
 \prod_{j=1}^N\Phi_0\bigl(\gamma(\tau_{j-1}), \gamma(\tau_j)\bigr) - [\gamma\|_0^1]^{-1} = \left( \prod_{j=1}^N J\bigl(\gamma(\tau_{j-1}), \gamma(\tau_j)\bigr)^{-1/2} - 1\right) [\gamma\|_0^1]^{-1}.
\end{equation}
  By Corollary~II.8.1 in \cite{ChavelRiemannian} and compactness of $M$, there exist constants $C_1, R>0$ such that  $|J(x, y) - 1| \leq C_1 d(x, y)^2$ for all $x, y \in M$ with $d(x, y) < R$. Hence for each $\alpha \in \R$, there exists a constant $C_\alpha$ such that 
  \begin{equation*}
    J(x, y)^{\alpha} \leq e^{C_\alpha d(x, y)^2}.
  \end{equation*}
  for such $x, y \in M$. Because $d(\gamma(\tau_{j-1}), \gamma(\tau_j)) = \Delta_j \tau d(x, y)$, we have
 \begin{equation*}
   \prod_{j=1}^N J\bigl(\gamma(\tau_{j-1}), \gamma(\tau_j)\bigr)^{\alpha} \leq e^{C_\alpha d(x, y)^2 \sum_{j=1}^N (\Delta_j\tau)^2} \leq e^{C_\alpha |\tau| d(x, y)^2}
 \end{equation*}
 Using this for $\alpha = \pm 1/2$ gives the lemma together with \eqref{PhiKnotAndJ}.
 \end{proof}

\begin{proof}[of Lemma~\ref{LemmaDeterminantOfEvaluationMap}]
Identify the tangent spaces $T_{\gamma(s)} M$ with $T_{\gamma(0)}M$ using parallel transport along $\gamma$ so that vector fields along $\gamma$ are identified with $T_{\gamma(0)}M$-valued functions. Let $\tau = \{0 = \tau_0 < \tau_1 < \dots < \tau_N = 1\}$, $N\geq 2$, be a partition of the interval $[0, 1]$ and write for abbreviation $\Delta_j := \Delta_j \tau = \tau_j-\tau_{j-1}$ throughout this proof.

{\itshape Step 1.}
Define the subspace $W_\tau \subset T_\gamma H_{xy}(M) = H^1_0([0, 1], \gamma^*TM)$ by
\begin{equation} \label{DefinitionWtau}
  W_\tau := \bigl\{ X \in T_\gamma H_{xy}(M) \mid X~\text{smooth on}~(\tau_{j-1}, \tau_j)~\text{with}~\nabla_s^2 X(s) = 0 \bigr\}.
\end{equation}
This means that elements $X \in W_\tau$ are piecewise linear, i.e.\ they have the form
\begin{equation} \label{FormOfXInWtau}
  X(\tau_{j-1}+s) = \left(1-\frac{s}{\Delta_j}\right)v_{j-1} + \frac{s}{\Delta_j} v_j,~~~~~~~~~~~~ v_j := X(\tau_{j}), ~~~ 0 \leq s \leq \Delta_j.
\end{equation}
Define
\begin{equation*}
  \Psi_\tau: \bigoplus_{j=1}^N T_{\gamma(\tau_j)}M \longrightarrow W_\tau,~~~~~~~~(v_1, \dots, v_{N-1}) \longmapsto X_v,
\end{equation*}
where $X_v$ is the unique element in $W_\tau$ with $X_v(\tau_j) = v_j$ (where we set $v_0 = v_N = 0$). Then by the explicit form \eqref{FormOfXInWtau} of $X_v = \Psi_\tau(v_1, \dots, v_{N-1})$, $X_w = \Psi_\tau(w_1, \dots, w_{N-1})$, we have (using the convention $v_0 = v_N = w_0 = w_N = 0$)
\begin{equation*}
\begin{aligned}
  \bigl(X_v, X_w\bigr)_{H^1} 
  &= \sum_{j=1}^N \int_{\tau_{j-1}}^{\tau_j} \< \frac{1}{\Delta_j}(v_j-v_{j-1}), \frac{1}{\Delta_j}(w_j-w_{j-1})\> \dd s\\
  &=  \sum_{j=1}^N \frac{1}{\Delta_j}\bigl(\<v_j, w_j\> + \<v_{j-1}, w_{j-1}\> - \<v_j, w_{j-1}\> - \<v_{j-1}, w_j\> \bigr)\\
  &= \< \begin{pmatrix}
     v_1 \\ \vdots \\ v_{N-1}
   \end{pmatrix}, D_\tau
   \begin{pmatrix}
     w_1 \\ \vdots \\ w_{N-1}
   \end{pmatrix}\>
\end{aligned}
\end{equation*}
where $D_\tau$ is the $n(N-1) \times n(N-1)$ matrix
\begin{equation*}
D_\tau := \begin{pmatrix}
     \left(\frac{1}{\Delta_1} + \frac{1}{\Delta_2}\right)\id & -\frac{1}{\Delta_2}\id & & \\
     -\frac{1}{\Delta_2}\id & \left(\frac{1}{\Delta_2} + \frac{1}{\Delta_3}\right)\id & \ddots &\\
     & \ddots & \ddots & -\frac{1}{\Delta_{N-1}}\id\\
     & & -\frac{1}{\Delta_{N-1}}\id & \left(\frac{1}{\Delta_{N-1}} + \frac{1}{\Delta_N}\right)\id
   \end{pmatrix}.
\end{equation*}
Using induction and Laplace's formula for determinants, one shows that $\det(D_\tau) = \prod_{j=1}^N \Delta_j^{-n}$. As a subspace of $H^1_0([0, 1], \gamma^*TM)$, $W_\tau$ carries the induced $H^1$ scalar product. With respect to this scalar product, we obtain that
\begin{equation}\label{DeterminantPsi}
  |\det(\Psi_\tau)| = \det(\Psi_\tau^*\Psi_\tau)^{1/2} = \det(D_\tau)^{1/2} = \prod_{j=1}^N \Delta_j^{-n/2}.
\end{equation}


{\itshape Step 2.} Define the operator
\begin{equation} \label{DefinitionTtau}
  K_\tau: W_\tau \longrightarrow T_\gamma H_{xy}(M), ~~~~~~~X \longmapsto K_\tau X,
\end{equation}
where $Y := K_\tau X$ is the unique solution of
\begin{equation*}
  \left\{\begin{aligned}-\nabla_s^2 Y(s) + \mathcal{R}_\gamma(s) Y(s) &= - \mathcal{R}_\gamma(s) X(s)~~~~~~ & & \text{for}~~~~ s \neq \tau_j\\
  ~~~Y(\tau_j) &= 0 & &\text{for}~~~~  j = 1, \dots, N,
  \end{aligned}\right.
\end{equation*}
with $\mathcal{R}_\gamma$ the curvature endomorphism along $\gamma$ considered in Section~\ref{SectionSobolevSpacesAlongPaths}.
This problem indeed has a unique solution, because $Y = K_\tau X$ is just patched together from the unique solutions of Dirichlet problems on each subinterval $[\tau_{j-1}, \tau_j]$. Namely, the self-adjoint operator $-\nabla_s^2 + \mathcal{R}_\gamma$ with Dirichlet boundary conditions is invertible on each of the subintervals $[\tau_{j-1}, \tau_j]$, because it has trivial kernel: Elements in the kernel are Jacobi fields with vanishing endpoints. A non-zero element in the kernel would therefore imply that $\gamma(\tau_{j-1})$ and $\gamma(\tau_j)$ are conjugate, which cannot happen for $N\geq 2$ as $\gamma$ is a minimizing geodesic. 

Because the right hand side is smooth in the interior on these subintervals, $Y$ is as well. For $X \in W_\tau$, set $\tilde{X} := X + K_\tau X := X+Y$. Then $\tilde{X} \in T_\gamma H_{xy;\tau}(M)$, because for $s \neq \tau_j$, we have
\begin{equation*}
  \nabla_s^2\tilde{X} = \underbrace{\nabla_s^2 X(s)}_{=0} + \nabla_s^2Y(s) = \mathcal{R}_\gamma(s) Y(s) + \mathcal{R}_\gamma(s) X(s) = \mathcal{R}_\gamma(s) \tilde{X}(s).
\end{equation*}
Thus $\tilde{X}$ is a piecewise Jacobi field, i.e.\ an element of $T_\gamma H_{xy;\tau}(M)$. Notice that
\begin{equation*}
  \id+K_\tau : W_\tau \longrightarrow T_\gamma H_{xy;\tau}(M)
\end{equation*}
is an isomorphism of vector spaces, because the dimensions coincide and it is injective: If $X = - K_\tau X$, for $X \in W_\tau$, then in particular $X(\tau_j) = -(K_\tau X)(\tau_j) = 0$ for all $j$, hence $X = 0$. Furthermore, for vectors $v_j \in T_{\gamma(\tau_j)}M$, $X:=(\id + K_\tau)\Psi_\tau(v_1, \dots, v_{N-1})$ is the piece-wise Jacobi field with $X(\tau_j) = v_j$. Therefore,
\begin{equation} \label{DecompositionDEvTau}
  \bigl(d \mathrm{ev}_\tau|_\gamma\bigr)^{-1} = (\id+K_\tau) \Psi_\tau.
\end{equation}
Extend $K_\tau$ to a bounded linear operator on $T_\gamma H_{xy}(M)$ through extension by zero on the orthogonal complement $W_\tau^\perp$. Denote by $i_\tau, p_\tau$ and $\iota_\tau$, $\pi_\tau$ the inclusions and orthogonal projections corresponding to the subspaces $W_\tau$ respectively $T_\gamma H_{xy;\tau}(M)$ of $T_\gamma H_{xy}(M)$.
Using \eqref{DecompositionDEvTau} and \eqref{DeterminantPsi}, we obtain
\begin{equation} \label{Calc1}
\begin{aligned}
  \bigl|\det\bigl(d \mathrm{ev}_\tau|_\gamma\bigr)\bigr|\prod_{j=1}^N {\Delta_j^{-n/2}}  &= \bigl|\det\bigl(\pi_\tau(\id+K_\tau)i_\tau\bigr)\bigr|^{-1}\frac{ \prod_{j=1}^N {\Delta_j^{-n/2}}}{\bigl|\det(\Psi_\tau)\bigr|}\\
  &= \bigl|\det\bigl(\pi_\tau(\id+K_\tau)i_\tau\bigr)\bigr|^{-1}
\end{aligned}
\end{equation}
Furthermore, 
\begin{equation} \label{Calc2}
\begin{aligned} 
\bigl|\det\bigl(\pi_\tau(\id+K_\tau)i_\tau\bigr)\bigr|^{} &= \det\bigl( p_\tau(\id+K_\tau)^*\iota_\tau \pi_\tau(\id+K_\tau)i_\tau\bigr)^{1/2}\\
  &= \det\bigl( p_\tau(\id+K_\tau)^*(\id+K_\tau)i_\tau\bigr)^{1/2},
\end{aligned}
\end{equation}
where in the last step we used that the image of $\id + K_\tau$ is contained in $T_\gamma H_{xy;\tau}(M)$ so that the projection and inclusion in the middle can be left out.
For $X_1, X_2 \in W_\tau$, let $Y_1 := K_\tau X_1$, $Y_2 := K_\tau X_2$ and calculate
\begin{equation*}
  (X_1, K_\tau X_2)_{H^1} = (X_1, Y_2)_{H^1} = \sum_{j=1}^N \int_{\tau_{j-1}}^{\tau_j} \< \nabla_s X_1(s), \nabla_s Y_2(s)\> \dd s = 0,
\end{equation*}
which follows from integration by parts since $\nabla_s^2 X_1 = 0$ for $s \in [\tau_{j-1}, \tau_j]$ and $Y_2(\tau_j) = Y_2(\tau_{j-1})= 0$ for all $j=1, \dots, N$. This shows $K_\tau X \subset W_\tau^\perp$. Put together, we get for $X_1, X_2 \in W_\tau$ that
\begin{equation*}
\begin{aligned}
  \bigl(X_1, (\id + &K_\tau)^*(\id + K_\tau)X_2\bigr)_{H^1} \\
  &= (X_1, X_2)_{H^1} +\underbrace{(X_1, K_\tau X_2)_{H^1}}_{=0} + \underbrace{(K_\tau X_1, X_2)_{H^1}}_{=0} + (K_\tau X_1, K_\tau X_2)_{H^1}\\
  &= \bigl(X_1, (\id + K_\tau^*K_\tau) X_2\bigr)_{H^1}, 
\end{aligned}
\end{equation*}
i.e.\ $p_\tau (\id + K_\tau)^*(\id + K_\tau)i_\tau = p_\tau (\id + K_\tau^*K_\tau)i_\tau$, and
\begin{equation} \label{DeterminantIdPlusTtau}
\begin{aligned}
  \det\bigl(p_\tau(\id + K_\tau)^*(\id + K_\tau)i_\tau\bigr)^{1/2} &= \det\bigl(p_\tau(\id +K_\tau^*K_\tau)i_\tau\bigr)^{1/2} \\
  &= \det\bigl(\id +K_\tau^*K_\tau\bigr)^{1/2},
\end{aligned}
\end{equation}
where the last determinant is a Fredholm determinant and the last step uses that $\id +K_\tau^*K_\tau$ has block diagonal form with respect to the decomposition $T_\gamma H_{xy}(M) = W_\tau \oplus W_\tau^\perp$.

Therefore, with a view on the  standard determinant estimate \eqref{StandardDeterminantEstimate}, we are led to estimate $\|K_\tau^*K_\tau\|_1 = \tr (K^*_\tau K_\tau) = \|K_\tau\|_{2}^2$, the Hilbert-Schmidt norm of $K_\tau$.

{\itshape Step 3.} We need some preliminary considerations. Let $[a, b]$ be any subinterval of $[0, 1]$ and write $P$ for the operator $-\nabla_s^2$ on $L^2([a, b], \gamma^*TM)$ with Dirichlet boundary conditions, as in Section~\ref{SectionSobolevSpacesAlongPaths}. Suppose that $[a, b] \subsetneq [0, 1]$. Then $P+\mathcal{R}_\gamma$ is an isomorphism from $H^m_0([a, b], \gamma^*TM)$ to $H^{m-2}_0([a, b], \gamma^*TM)$ for each $m \in \R$ (remember that $\gamma$ is a minimizing geodesic, hence $\gamma|_{[a, b]}$ is unique minimizing, so there are no non-trivial Jacobi fields with vanishing endpoints along $\gamma|_{[a, b]}$, i.e.\ the kernel of $P + \mathcal{R}_\gamma$ is trivial). We now show that
\begin{equation} \label{OpNormBoundOnInverse}
  \bigl\|(P+\mathcal{R}_\gamma)^{-1}X\bigr\|_{H^1} \leq \frac{(b-a)^2}{\pi^2 - \|\mathcal{R}_\gamma\|(b-a)^2} \|X\|_{H^{1}},
\end{equation}
for each $X \in H^1([a, b], \gamma^*TM)$ and any $\gamma \in \Gamma_{xy}^{\min}$, where $\|\mathcal{R}_\gamma\|$ is the operator norm of the operator $X \mapsto \mathcal{R}_\gamma X$ on $H^1_0([0, 1], \gamma^*TM)$. First we have using Lemma~\ref{LemmaOperatorNormOfP} above that
\begin{equation*}
  \|P^{-1} \mathcal{R}_\gamma X\|_{H^1} \leq \frac{(b-a)^2}{\pi^2} \|\mathcal{R}_\gamma X\|_{H^1} \leq \frac{(b-a)^2}{\pi^2}\|\mathcal{R}_\gamma\| \|X\|_{H^1},
\end{equation*}
since the operator norm of $\mathcal{R}_\gamma$ on $[a, b]$ is less or equal to the operator norm of $\mathcal{R}_\gamma$ on the interval $[0, 1]$.
We find for all $X \in H^1_0([a, b], \gamma^*TM)$ that
\begin{equation*}
  \bigl\|(\id + P^{-1} \mathcal{R}_\gamma)X\bigr\|_{H^1} \geq \|X\|_{H^1} - \|P^{-1}\mathcal{R}_\gamma X\|_{H^1} \geq \left(1 - \|\mathcal{R}_\gamma\|\frac{(b-a)^2}{\pi^2}\right)\|X\|^2_{H^1}.
\end{equation*}
Because $\id + P^{-1} \mathcal{R}_\gamma$ is self-adjoint on $H^1_0([a, b], \gamma^*TM)$ as is easy to verify, we obtain for its smallest eigenvalue
\begin{equation*}
  \mu_{\min} = \inf_{X \neq 0} \frac{\|(\id + P^{-1} \mathcal{R}_\gamma)X\|_{H^1}}{\|X\|_{H^1}} \geq \left(1 - \|\mathcal{R}_\gamma\|\frac{(b-a)^2}{\pi^2}\right).
\end{equation*}
The spectral radius of the inverse $(\id + P^{-1} \mathcal{R}_\gamma)^{-1}$ is equal to $1/\mu_{\min}$. Since $\id + P^{-1} \mathcal{R}_\gamma$ is self-adjoint on $H^1_0([a, b], \gamma^*TM)$ and so is its inverse, the spectral radius equals the operator norm, whence
\begin{equation*}
  \bigl\|(\id + P^{-1} \mathcal{R}_\gamma)^{-1}X\bigr\|_{H^1} \leq \frac{1}{\mu_{\min}} \|X\|_{L^2} \leq \frac{\pi^2}{\pi^2 - \|\mathcal{R}_\gamma\|(b-a)^2}\|X\|_{H^1}
\end{equation*}
Finally, using Lemma~\ref{LemmaOperatorNormOfP} again, we get
\begin{equation*}
\begin{aligned}
  \bigl\|(P+\mathcal{R}_\gamma)^{-1}X\bigr\|_{H^1} &= \bigl\|P^{-1}(\id+P^{-1}\mathcal{R}_\gamma)^{-1}X\bigr\|_{H^1}\\
  &\leq \frac{(b-a)^2}{\pi^2}\bigl\|(\id+P^{-1}\mathcal{R}_\gamma )^{-1}X\bigr\|_{H^1}\\
  &\leq \frac{(b-a)^2}{\pi^2 - \|\mathcal{R}_\gamma\|(b-a)^2}\|X\|_{H^1},
\end{aligned}
\end{equation*}
which is the claim.

{\itshape Step 4.} We finally derive a bound on $\|K_\tau\|_{2}^2$. For any vector $X \in T_\gamma H_{xy;\tau}(M)$ and any $j=1, \dots, N$, we have $K_\tau X|_{[\tau_{j-1}, \tau_j]} = - (P+\mathcal{R}_\gamma)^{-1} \mathcal{R}_\gamma X|_{[\tau_{j-1}, \tau_j]}$, where $(P+\mathcal{R}_\gamma)^{-1}$ is the operator discussed in Step~3 on the interval $[a, b] := [\tau_{j-1}, \tau_j]$.

Let $E_1, E_2, \dots, E_{n(N-1)}$ be an orthonormal basis of $W_\tau$.
Using the estimate \eqref{OpNormBoundOnInverse} from Step~3 on the operator norm of $(P+\mathcal{R}_\gamma)^{-1}$ on $H^1([\tau_{j-1}, \tau_j], \gamma^*TM)$, we obtain
\begin{equation*}
\begin{aligned}
  \|K_\tau\|_{2}^2 &= \sum_{i=1}^{n(N-1)} \|K_\tau E_i\|_{H^1}^2 = \sum_{i=1}^{n(N-1)} \sum_{j=1}^N \bigl\|K_\tau E_i|_{[\tau_{j-1}, \tau_j]}\bigr\|_{H^1}^2 \\
  &= \sum_{i=1}^{n(N-1)} \sum_{j=1}^N \bigl\|-(P+\mathcal{R}_\gamma)^{-1} \mathcal{R}_\gamma E_i|_{[\tau_{j-1}, \tau_j]}\bigr\|_{H^1}^2\\
  &\leq \sum_{i=1}^{n(N-1)} \sum_{j=1}^N \left(\frac{\Delta_j^2}{\pi^2 - \|\mathcal{R}_\gamma\|\Delta_j^2}\right)^2\bigl\|\mathcal{R}_\gamma E_i|_{[\tau_{j-1}, \tau_j]}\bigr\|_{H^1}^2\\
  &\leq \sum_{i=1}^{n(N-1)} \left(\frac{|\tau|^2}{\pi^2 - \|\mathcal{R}_\gamma\||\tau|^2}\right)^2\bigl\|\mathcal{R}_\gamma E_i\bigr\|_{H^1}^2\leq n(N-1) \left( \frac{\|\mathcal{R}_\gamma\| |\tau|^2}{ \pi^2 - \|\mathcal{R}_\gamma\||\tau|^2} \right)^2 
\end{aligned} 
\end{equation*}
We now suppose that $|\tau| \leq C/N$ for some $C>0$. 
Suppose additionally the partition $\tau$ be so fine that $|\tau| \leq \pi/\sqrt{2 \|\mathcal{R}_\gamma\|}$, or equivalently  $\pi^2 - \|\mathcal{R}_\gamma\||\tau|^2 \geq \pi^2/2$. By the assumption $|\tau| \leq C/N$, this is the case in particular if $N \geq N_0 := \lceil C \sqrt{2\|\mathcal{R}_\gamma\|}/ \pi\rceil$. For such $\tau$, we have
\begin{equation*}
   \frac{ \|\mathcal{R}_\gamma\| |\tau|^2}{\pi^2 - \|\mathcal{R}_\gamma\||\tau|}  \leq \frac{2\|\mathcal{R}_\gamma\||\tau|^2}{\pi^2} \leq  \frac{2\|\mathcal{R}_\gamma\|C^2}{\pi^2N^2} = \frac{N_0^2}{N^2}
\end{equation*}
and
\begin{equation*}
\|K_\tau\|_{2}^2 \leq n(N-1)\left(\frac{N_0^2}{N^2}\right)^2 \leq {n N_0^2} \frac{1}{N^3}.
\end{equation*}
With a view on the chain of calculations \eqref{Calc1}-\eqref{DeterminantIdPlusTtau}, this concludes the proof  using the determinant estimate \eqref{StandardDeterminantEstimate}, because the operator norm $\|\mathcal{R}_\gamma\|$ is uniformly bounded for $\gamma \in \Gamma_{xy}^{\min}$.
\end{proof}

\begin{remark}
Notice that if $M$ is flat, we have $W_\tau = H_{xy;\tau}(M)$ and the operator $K_\tau$ of the above proof is zero. Hence in the flat case, we have
\begin{equation*}
 \bigl|\det \bigl(d \mathrm{ev}_\tau|_\gamma \bigr)\bigr|  \prod_{j=1}^N (\Delta_j\tau)^{-n/2}\equiv 1,
\end{equation*}
for each partition $\tau$.
\end{remark}

\begin{proof}[of Lemma~\ref{LemmaTangentSpacesDense}] The proof is divided into two steps.

{\em Step 1.} We first show that the union of the spaces $W_\tau$ for $\tau \in S$ is dense in the space $H^1_0([0, 1], \gamma^*TM)$, where $W_\tau$ is the space defined in \eqref{DefinitionWtau}. Namely, we claim that given a partition $\tau=\{ 0=\tau_0 < \tau_1 <\dots < \tau_N = 1\}$, a vector field $X \in H^1_0([0, 1], \gamma^*TM)$ is in the orthogonal complement of $W_\tau$ if only if $X(\tau_j) = 0$ for all $j=1, \dots, N-1$. Indeed, for a given $v \in T_{\gamma(\tau_j)}M$, define $Y \in W_\tau$ by
\begin{equation*}
  Y(s) = \begin{cases} s(1-\tau_j) v & s \leq \tau_j \\ (1-s)\tau_j v & s \geq \tau_j,\end{cases}
\end{equation*}
where we identified the spaces $T_{\gamma(s)}M$ by parallel transport along $\gamma$. Then integrating by parts and using that $\nabla_s^2 X \equiv 0$ on $(\tau_{j-1}, \tau_j)$ yields
\begin{equation*}
  (X, Y)_{H^1} = \sum_{j=1}^N \int_{\tau_{j-1}}^{\tau_j}\!\! \< \nabla_s X(s), \nabla_s Y(s)\> \dd s = \sum_{j=1}^{N-1} \< X(\tau_j), Y(\tau_j-)-Y(\tau_j+)\> = \< X(\tau_j), v\>.
\end{equation*}
This proves the claim, since this scalar product is zero for all $v$ chosen this way if and only if $X(\tau_j) = 0$ for all $j$. 

Now suppose that $X \in H_{xy;\tau}(M)$ is in the orthogonal complement of $W_\tau$ for all $\tau \in S$. Then by the observation before, we obtain that necessarily $X(s)=0$ for all $s \in [0, 1]$ for which there exists a partition $\tau \in S$ with $s \in \tau$. Because of the condition on the set $S$, the set of such $s$ in dense in $[0, 1]$, so from continuity follows $X \equiv 0$. Therefore the union of all $W_\tau$, $\tau \in S$ must be a dense subset of $H_{xy}(M)$.

{\em Step 2.} Suppose that $W_\tau \neq H_{xy;\tau}(M)$, i.e.\ $\mathcal{R}_\gamma \neq 0$ (otherwise, we are already done with the proof). Let $Y \in W_\tau$. Then if $K_\tau$ is the operator defined in \eqref{DefinitionTtau}, then $Y + K_\tau Y \in T_\gamma H_{xy;\tau}(M)$, as seen in Step 2 of the proof of Lemma~\ref{LemmaDeterminantOfEvaluationMap} above. By \eqref{OpNormBoundOnInverse}, we have
\begin{equation*}
\begin{aligned}
  \|K_\tau Y\|_{H^1}^2 &= \sum_{j=1}^N \bigl\|-(P+\mathcal{R}_\gamma)^{-1}\mathcal{R}_\gamma Y|_{\tau_{j-1}, \tau_j}\bigr\|_{H^1}^2 \leq \sum_{j=1}^N \left(\frac{\Delta_j^2}{\pi^2 - \|\mathcal{R}_\gamma\| \Delta_j^2}\right)^2 \|\mathcal{R}_\gamma Y|_{\tau_{j-1}, \tau_j}\|_{H^1}^2 \\
  &\leq |\tau|^4 \frac{4}{\pi^4} \|\mathcal{R}_\gamma Y\|_{H^1}^2 \leq |\tau|^4 \frac{4}{\pi^4} \|\mathcal{R}_\gamma\|^2 \|Y\|_{H^1}^2
\end{aligned}
\end{equation*}
whenever $\pi^2 - \|\mathcal{R}_\gamma\| |\tau|^2 \leq \pi^2/2$, or equivalently $|\tau| \leq \pi / \sqrt{2\|\mathcal{R}_\gamma\|}$ (here $\|\mathcal{R}_\gamma\|$ is the operator norm of the operator $X \mapsto \mathcal{R}_\gamma X$ on $H^1_0([0, 1], \gamma^*TM)$. We conclude that the operator norm of the operators $K_\tau$ for $|\tau|$ small enough satisfies the bound $\|K_\tau\| \leq C |\tau|^2$ with a constant $C>0$ independent of $\tau$. Hence
\begin{equation*}
\begin{aligned}
  \bigl\|X - (Y + K_\tau Y)\bigr\|_{H^1} &\leq \|X-Y\|_{H^1} + \|K_\tau Y\|_{H^1} 
  \leq \|X-Y\|_{H^1} + \|K_\tau\| \|Y\|_{H^1}\\
  &\leq \|X-Y\|_{H^1} +C|\tau|^2\bigl(\|X-Y\|_{H^1} + \|X\|_{H^1}\bigr).
\end{aligned}
\end{equation*}
Now given $\varepsilon >0$, choose $\delta>0$ such that
\begin{equation*}
  \delta^2 < \min \left\{ \frac{\varepsilon}{C \bigl(\varepsilon + 2\|X\|_{H^1}\bigr)}, \frac{\pi^2}{2 \|\mathcal{R}_\gamma\|} \right\}
\end{equation*}
and let $S^\prime \subset S$ be the set containing all partitions $\tau \in S$ with $|\tau| \leq \delta$. Then $S^\prime$ still has the property from the lemma, so by Step 1, for some $\tau \in S^\prime$, we find $Y \in W_\tau$  such that $\|X - Y\|_{H^1} < \varepsilon /2$. Then by the choice of $\delta$, if $|\tau| \leq \delta$, we have $\|X - (Y + K_\tau Y)\bigr\|_{H^1} \leq \varepsilon$. Because $\varepsilon$ was arbitrary and $Y + K_\tau Y \in H_{xy;\tau}(M)$, $\tau \in S$, this shows that the union of all $H_{xy;\tau}(M)$, $\tau \in S$ is dense in $H^1_0([0, 1], \gamma^*TM)$.
\end{proof}

\section{Zeta Determinants and the $L^2$-picture} \label{SectionZetaDeterminants}

So far, we have seen that in the case that the set $\Gamma_{xy}^{\min}$ of minimizing geodesics between the points $x$, $y$ is a $k$-dimensional non-degenerate submanifold of $H_{xy}(M)$ (with respect to the action functional), we have
\begin{equation*}
 \lim_{t \rightarrow 0}~(4 \pi t)^{k/2} \frac{p_t^L(x, y)}{\e_t(x, y)} = \int_{\Gamma_{xy}^{\min}}\frac{[\gamma\|_0^1]^{-1}}{\det\bigl(\nabla^2 S|_{N_\gamma \Gamma_{xy}^{\mathrm{min}}}\bigr)^{1/2}} \,\dd^{H^1} \gamma.
\end{equation*}
The expression on the right hand side depends on the choice of a Riemannian metric on the manifold $H_{xy}(M)$ in two ways: First, because we integrate over the submanifold $\Gamma_{xy}^{\min}$ using the Riemannian volume density of the induced metric. Secondly, because we take the determinant of the bilinear form $\nabla^2 S|_{N_\gamma \Gamma_{xy}^{\min}}$ using the metric on $N_\gamma \Gamma_{xy}^{\min}$ (because to calculate the determinant of a bilinear form, we need a metric). In both cases, the $H^1$ metric \eqref{H1Metric} turned out to be the correct choice.

\medskip

However, there is another possible choice for the determinant of an operator on an infinite-dimensional space: the zeta determinant, which is defined for a certain class of unbounded operators on a Hilbert space. This approach is often used in physics to assign finite values to otherwise ill-defined path integrals, see e.g.\ \cite{Hawking} or \cite{WittenDiracOperators}. In our situation, the equality $\nabla^2 S|_\gamma[X, Y] = (X, (-\nabla_s^2+\mathcal{R}_\gamma)Y)_{L^2}$ (see \eqref{HessianEnergyOnL2}) motivates to replace the determinant of $\nabla^2 S|_\gamma$ by the zeta determinant of the Jacobi-operator $-\nabla_s^2 + \mathcal{R}_\gamma$. 

This determinant does not depend on the choice of a Sobolev metric on the path spaces. Instead, it only depends on the eigenvalues of $-\nabla_s^2 + \mathcal{R}_\gamma$, considered as an unbounded operator on the Hilbert space $L^2([0, 1], \gamma^*TM)$. Since the $H^1$ metric on $H_{xy}(M)$ does no longer play a role then, it seems that one should also equip $\Gamma_{xy}^{\min}$ with another metric when performing the integral. Here the $L^2$ metric seems to be the natural choice.

\medskip

For an elliptic non-negative self-adjoint pseudo-differential operator $P$ of order $d>0$, acting on an $m$-dimensional compact manifold $\Sigma$, the {\em zeta function} $\zeta_P$ is defined by
\begin{equation} \label{ZetaDefinition}
  \zeta_P(z) := \sum_{\lambda\neq 0} \lambda^{-z},
\end{equation}
where the sum runs over all non-zero eigenvalues $\lambda$ of $P$. Here, $\Sigma$ may have a boundary, in which case we need to give appropriate boundary conditions to $P$. This sum converges for $\mathrm{Re}(z) > m/d$; however, one can check that $\zeta_P$ possesses a meromorphic extension to all of $\C$ and that zero is not a pole \cite[Section 1.12]{gilkey95}. Therefore, one can define the {\em zeta-regularized determinant}
\begin{equation*}
  \det\nolimits_\zeta(P) := e^{-\zeta^\prime_P(0)}.
\end{equation*}
If $P$ actually has zero eigenvalues that were left out in the sum \eqref{ZetaDefinition}, it is conventional to write $\det\nolimits_\zeta^\prime(P)$ instead.
The definition is motivated by the fact that if one (formally!) plugs the series \eqref{ZetaDefinition} into the right hand side of this definition (which is not possible since one cannot evaluate it at zero), one obtains
\begin{equation*}
  e^{-\zeta^\prime_P(0)} ~~\stackrel{\text{formally}}{=} ~~\prod_{\lambda \neq 0} \lambda,
\end{equation*}
the product of the non-zero eigenvalues, which of course diverges; the zeta determinant $-$ somewhat magically $-$ assigns a finite value to this product.


\begin{example}[Dirichlet-Laplacian along a Geodesic] \label{ExampleZetaLaplacian}
Let $\gamma$ be a smooth path in an $n$-dimensional Riemannian manifold $M$ parametrized by $[0, t]$. Already in Section~\ref{SectionSobolevSpacesAlongPaths}, we found the eigenvalues of the operator $P= -\nabla_s^2$ with Dirichlet boundary conditions on the space $L^2([0, t], \gamma^*TM)$ to be the numbers $\lambda_k = \pi^2 k^2/t^2$, each of multiplicity $n$. Hence for $\mathrm{Re}\,z>1/2$, we have
\begin{equation*}
  \zeta_P(z) = n \sum_{k=1}^\infty \left(\frac{\pi^2k^2}{t^2}\right)^{-z} = n\frac{t^{2z}}{\pi^{2z}} \sum_{k=1}^\infty k^{-2z} = n\frac{t^{2z}}{\pi^{2z}} \zeta(2z),
\end{equation*}
where $\zeta$ without subscript denotes the usual Riemann zeta function. Therefore,
\begin{equation*}
\begin{aligned}
  \zeta_P^\prime(0) &= 2n \bigl( \log(t)-\log(\pi)\bigr)\zeta(0) + 2n\zeta^\prime(0) = -n \log(2t)
\end{aligned}
\end{equation*}
as it is well known that $\zeta(0) = -1/2$ and $\zeta^\prime(0) = -\log(2\pi)/2$ \cite{SondowRiemannValues}. We obtain
\begin{equation}
  \det\nolimits_\zeta(-\nabla_s^2) = e^{-\zeta_P^\prime(0)} = (2t)^n
\end{equation}
for the zeta determinant. For the operators $P^m$, $m>0$, one easily sees that $\zeta_{P^m}(z) = \zeta_P(mz)$, hence $\det\nolimits_\zeta(P^m) = \det\nolimits_\zeta(P)^m$.
\end{example}

More generally, the zeta determinant can be defined for a wide class of (necessarily unbounded) closed operators with discrete spectrum on an abstract Hilbert space $H$, called {\em zeta-admissible} (for the definition, see \cite[Section 2]{ScottDeterminants}). That an operator is zeta-admissible essentially means that it has a well-defined zeta function which does not have a pole at zero. We will not need the exact definition here (which is somewhat involved); we will only need that Laplace type operators $P$ on intervals with Dirichlet boundary conditions are zeta-admissible, as well as their positive powers. Such operators $P$ are well-known to be zeta-admissible; this can be shown e.g.\ using the heat trace expansion as in \cite[Section~10]{gilkey95}. 

 The following result then generates many more examples.

\begin{proposition}[Multiplicativity]{\normalfont \cite[Thm.~2.18]{ScottDeterminants}} \label{PropScott}
Let $\mathcal{H}$ be a Hilbert space, let $P$ be a closed and invertible operator on $\mathcal{H}$ with positive spectrum and
let $T := \id + W$ with $W$ trace-class on $\mathcal{H}$. If $P$ is zeta-admissible, then so are $PT$ and $TP$ and we have
\begin{equation*}
  \det\nolimits_\zeta(PT) = \det\nolimits_\zeta(TP) =  \det\nolimits_\zeta(P) \det(T),
\end{equation*}
where $\det(T)$ denotes the usual Fredholm determinant.
\end{proposition}

\begin{remark}
We generally have $\det\nolimits_\zeta(AB) \neq \det\nolimits_\zeta(A)\det\nolimits_\zeta(B)$. The above result is then the correct replacement for this product rule.
\end{remark}

\begin{corollary}[Zeta Relativity] \label{CorollaryZetaRelativity}
Let $P_1, P_2$ be positive self-adjoint Laplace type operators with Dirichlet boundary conditions on the interval $[0, t]$, acting on the bundle $\gamma^*TM$, where $\gamma$ is a smooth path in some Riemannian manifold $M$. Suppose that the difference $P_1 - P_2$ is of order zero and that $P_1$ and $P_2$ have trivial kernels. Then $P_1^{-1} P_2$ is well defined and determinant-class on $L^2([0, t], \gamma^*TM)$ and we have
\begin{equation*}
  \det(P_1^{-1}P_2) = \frac{\det\nolimits_\zeta(P_2)}{\det\nolimits_\zeta(P_1)},
\end{equation*}
where the left hand side is the usual Fredholm determinant.
\end{corollary}

\begin{proof}
Because $P_1$ has trivial kernel, its inverse $P_1^{-1}$ is well defined by spectral calculus, and $P_1^{-1}: L^2([0, t], \gamma^*TM) \longrightarrow H^2_0([0, t], \gamma^*TM)$ is a bounded operator. By Lemma~\ref{LemmaInclusionHS}, the inclusion $H^2_0([0, t], \gamma^*TM) \longrightarrow L^2([0, t], \gamma^*TM)$ is nuclear; hence the operator $P_1^{-1}: L^2([0, t], \gamma^*TM) \longrightarrow L^2([0, t], \gamma^*TM)$ is trace-class, because it can be written as the composition of a bounded operator and a nuclear operator.
  
Write $P_2 = P_1 + V$ for an endomorphism field $V \in C^\infty([0, t], \gamma^*TM)$. Then
   \begin{equation*}
     P_1^{-1}P_2 = P_1^{-1}(P_1 + V) = \id + P_1^{-1}V
   \end{equation*}
   is determinant-class, because $P_1^{-1} V$ is trace-class. We can now apply Prop.~\ref{PropScott} on the Hilbert space $L^2([0, t], \R^n)$ with $P=P_1$ and $T= P_1^{-1}P_2$ to obtain the required determinant identity.
\end{proof}

Similarly, the following is true.

\begin{proposition} \label{PropEnergyAsZetaQuotient}
Let $M$ be a Riemannian manifold and let $(x, y) \in M \bowtie M$. Then we have
\begin{equation*}
  \det\bigl(\nabla^2 S|_{\gamma_{xy}}\bigr) = \frac{\det\nolimits_\zeta(-\nabla_s^2 + \mathcal{R}_{\gamma_{xy}})}{\det\nolimits_\zeta(-\nabla_s^2)} = 2^{-n} \det\nolimits_\zeta(-\nabla_s^2 + \mathcal{R}_{\gamma_{xy}}),
\end{equation*}
where $\gamma_{xy}$ is the unique minimizing geodesic travelling from $x$ to $y$ in time one and $-\nabla_s^2 + \mathcal{R}_{\gamma_{xy}}$ is the {\em Jacobi operator} as in Section~\ref{SectionSobolevSpacesAlongPaths}. Both operators on the right hand side carry Dirichlet boundary conditions.
\end{proposition}

Combining this with Corollary~\ref{CorollaryJacobian} and Example~\ref{ExampleZetaLaplacian}, we may express the Jacobian of the exponential map as the zeta determinant of the Jacobi operator.

\begin{corollary}
Let $M$ be a Riemannian manifold and $(x, y) \in M \bowtie M$. Then for any $t>0$,
\begin{equation*}
  J(x, y) =  \frac{\det\nolimits_\zeta(-\nabla_s^2 + \mathcal{R}_{\gamma_{xy}})}{\det\nolimits_\zeta(-\nabla_s^2)} ,
\end{equation*}
where $J(x, y)$ denotes the Jacobian of the exponential map, as in Remark~\eqref{JacobianOfExponentialMap}.
\end{corollary}

\begin{proof}[of Prop.~\ref{PropEnergyAsZetaQuotient}]
Write $P:= -\nabla_s^2$ and $\gamma := \gamma_{xy}$ for abbreviation. By \eqref{HessianEnergyOnH1}, we have
\begin{equation*}
  \nabla^2 S|_{\gamma}[X, Y] = \bigl(X, P^{-1}(P + \mathcal{R}_\gamma)Y\bigr)_{H^1}.
\end{equation*}
Set $T := P^{-1}(P + \mathcal{R}_\gamma)$. Because $P^{-1/2}: L^2([0, t], \gamma^*TM) \rightarrow H^1_0([0, t], \gamma^*TM)$ is an isometry, we have
\begin{equation*}
  \det\bigl(\nabla^2 S|_{\gamma}\bigr) = \det\nolimits^{H^1}\bigl(T\bigr) = \det\nolimits^{L^2}\bigl(P^{1/2}TP^{-1/2}\bigr) = \det\nolimits^{L^2}\bigl(P^{-1/2}(P+\mathcal{R}_\gamma)P^{-1/2}\bigr).
\end{equation*}
The operator $P^{-1/2}(P+\mathcal{R}_\gamma)P^{-1/2}$ is indeed determinant-class, since
\begin{equation*}
  P^{-1/2}(P + \mathcal{R}_\gamma)P^{-1/2} = \id + P^{-1/2} \mathcal{R}_\gamma P^{-1/2} =: \id + \tilde{W},
\end{equation*}
where $\tilde{W}$ is the composition of two Hilbert-Schmidt operators and a bounded operator, hence trace-class. Set $W := P^{-1} \mathcal{R}_\gamma$. Then by Prop.~\ref{PropScott},
\begin{equation*}
\begin{aligned}
  \det\nolimits^{L^2}\bigl(\id + \tilde{W}\bigr)\det\nolimits_\zeta\bigl(P^{1/2}\bigr) &= \det\nolimits_\zeta((\id+\tilde{W})P^{1/2}\bigr) = \det\nolimits_\zeta(P^{1/2}(\id+{W})\bigr) \\
  &= \det\nolimits_\zeta\bigl(P^{1/2}\bigr)\det\nolimits^{L^2}\bigl(\id + {W}\bigr),
\end{aligned}
\end{equation*}
since $P^{1/2}$ is zeta-admissible. This shows that the $L^2$-determinant of $\id + \tilde{W}$ is equal to the $L^2$-determinant of $\id + W = P^{-1}(P + \mathcal{R}_\gamma)$ (the latter now being an operator on $L^2([0, 1], \gamma^*TM)$!). The result now follows from Corollary~\ref{CorollaryZetaRelativity}.
\end{proof}

In the particular case $(x, y) \in M \bowtie M$, we obtain
\begin{equation*}
  \lim_{t\rightarrow 0}~\frac{p_t^L(x, y)}{\e_t(x, y)} = \frac{\det\nolimits_\zeta^\prime(-\nabla_s^2 + \mathcal{R}_\gamma)^{-1/2}}{\det\nolimits_\zeta(-\nabla_s^2)^{-1/2}} [\gamma\|_0^1]^{-1}.
\end{equation*}
Now we prove the general case, Thm.~\ref{ThmL2Picture}. 

\begin{proof}[of Thm.~\ref{ThmL2Picture}]
By Thm.~\ref{ThmH1Picture}, we have
\begin{equation*}
 \lim_{t\rightarrow 0}~ (4\pi t)^{k/2}\,\frac{p_t^L(x, y)}{\e_t(x, y)}= \int_{\Gamma_{xy}^{\min}} \frac{[{\gamma}\|_0^1]^{-1}}{\det\bigl(\nabla^2 S|_{N_{{\gamma}} \Gamma_{xy}^{\mathrm{min}}}\bigr)^{1/2}} \dd^{{H}^1} {\gamma},
\end{equation*}
when $\Gamma_{xy}^{\min}$ is endowed with the $H^1$ metric \eqref{H1Metric}. By the transformation formula, we have
\begin{equation} \label{SecondLastFormulaPhiZero}
  \lim_{t\rightarrow 0}~ (4\pi t)^{k/2}\,\frac{p_t^L(x, y)}{\e_t(x, y)} = \int_{\Gamma_{xy}^{\min}} \frac{[{\gamma}\|_0^1]^{-1}}{\det\bigl(\nabla^2 S|_{N_{{\gamma}} \Gamma_{xy}^{\mathrm{min}}}\bigr)^{1/2}\det\bigl(d \id|_\gamma\bigr)} \dd^{{L}^2} {\gamma},
\end{equation} 
where $\det(d \id|_\gamma)$ denotes the determinant of the identity map from $\Gamma_{xy}^{\min}$ with the $H^1$ metric to the same space with the $L^2$ metric. Fix $\gamma \in \Gamma_{xy}^{\min}$ and let $f_1, \dots, f_k$ be an $H^1$-orthonormal basis of $T_\gamma \Gamma_{xy}^{\min} \cong \ker(P+\mathcal{R}_\gamma)$. Then
\begin{equation} \label{Determinant1} 
  \det\bigl(d \id|_\gamma\bigr) = \det\Bigl((f_i, f_j)_{{L}^2}\Bigr)_{1 \leq i, j \leq k}^{1/2}.
\end{equation}
Notice that $f_1, \dots, f_k$ are smooth by elliptic regularity. Let $f_{k+1}, f_{k+2}, \dots$ be a smooth $H^1$-orthonormal basis of $N_\gamma \Gamma_{xy}^{\min}$. By Thm.~\ref{ThmDeterminantConvergence} (respectively Remark~\ref{RemarkDeterminantConvergence}) and \eqref{HessianEnergyOnL2}, we have
\begin{equation} \label{Determinant2}
  \det\bigl(\nabla^2S|_{N_\gamma \Gamma_{xy}^{\min}}\bigr) = \lim_{N\rightarrow \infty} \det \Bigl( \bigl(f_i, (P+\mathcal{R}_\gamma)f_j\bigr)_{L^2}\Bigr)_{k+1 \leq i, j \leq N}.
\end{equation}
Let $\Pi$ be the $H^1$-orthogonal projection in $H^1_0([0, 1], \gamma^* TM)$ onto $\mathrm{ker}(P+\mathcal{R}_\gamma)$. Because $\Pi$ has finite rank, it is bounded with respect to the $L^2$ norm and therefore extends uniquely to a bounded operator on $L^2([0, 1], \gamma^*TM)$, which is still a projection onto $\mathrm{ker}(P+\mathcal{R}_\gamma)$ (since it is idempotent), but not necessary an orthogonal projection.
Set $Q := P+\mathcal{R}_\gamma + \Pi$. Then $Q$ is zeta-admissible by Prop.~\ref{PropScott} because it can be written in the form $Q = P(\id+W)$ with $W = P^{-1}(\mathcal{R}_\gamma + \Pi)$, which is trace-class by Lemma~\ref{LemmaInclusionHS}. Hence $Q$ is zeta-admissible. 

With respect to the orthogonal basis $f_1, f_2, \dots$ of the space $H^1_0([0, 1], \gamma^*TM)$ used above, we have
\begin{equation*}
  (f_i, Q f_j)_{L^2} = \begin{cases} (f_i, f_j)_{L^2} & \text{if}~1 \leq i, j \leq k\\
  \bigl(f_i, (P+\mathcal{R}_\gamma)f_j\bigr)_{L^2} & \text{if}~i, j > k\\
  0 & \text{if}~1 \leq i \leq k~\text{and}~j > k.
  \end{cases}
\end{equation*}
To see that third case, if $1 \leq i \leq k$ and $j>k$, calculate
\begin{equation*}
  (f_i, Q f_j)_{L^2} = \bigl(f_i, (P+\mathcal{R}_\gamma)f_j\bigr)_{L^2} + (f_i, \Pi f_j)_{L^2} = \bigl((P+\mathcal{R}_\gamma)f_i, f_j\bigr)_{L^2} = 0.
\end{equation*}
Hence the infinite matrix with entries $(f_i, Q f_j)_{L^2}$ is block triangular with respect to the orthogonal splitting of $H^1_0([0, 1], \gamma^* TM)$ into $\mathrm{ker}(P+\mathcal{R}_\gamma)$ and its orthogonal complement, and we have
\begin{equation*}
  \det \Bigl((f_i, Q f_j)_{L^2}\Bigr)_{1 \leq i, j \leq N} = \det\Bigl((f_i, f_j)_{{L}^2}\Bigr)_{1 \leq i, j \leq k} \det \Bigl( \bigl(f_i, (P+\mathcal{R}_\gamma)f_j\bigr)_{L^2}\Bigr)_{k+1 \leq i, j \leq N}.
\end{equation*}
for all $N > k$.
Plugging in \eqref{Determinant1} and \eqref{Determinant2}, we then obtain
\begin{equation*}
\begin{aligned}
  \det\bigl(\nabla^2S|_{N_\gamma \Gamma_{xy}^{\min}}\bigr)^{1/2}\det\bigl(d \id|_\gamma\bigr) 
  &= \lim_{N \rightarrow \infty} \det\Bigl((f_i, Q f_j)_{{L}^2}\Bigr)_{1 \leq i, j \leq N}^{1/2}\\
  &= \lim_{N\rightarrow \infty} \det\Bigl((f_i, P^{-1}Q f_j)_{H^1}\Bigr)_{1 \leq i, j \leq N}^{1/2}\\
  &= \det\nolimits^{H^1}(P^{-1}Q)^{1/2}.
\end{aligned}
\end{equation*}
Because $P^{-1/2}: L^2([0, 1], \gamma^*TM)\longrightarrow H^1_0([0, 1], \gamma^*TM)$ is an isometry, we obtain
\begin{equation*}
  \det\nolimits^{H^1}(P^{-1}Q) = \det\nolimits^{L^2}(P^{-1/2}QP^{-1/2}).
\end{equation*}
Again, we have by Prop.~\ref{PropScott},
\begin{equation*}
   \det\nolimits^{L^2}(P^{-1/2}QP^{-1/2})\det\nolimits_\zeta(P^{1/2}) = \det\nolimits_\zeta(P^{-1/2}Q) = \det\nolimits_\zeta(P^{1/2})\det\nolimits^{L^2}(P^{-1}Q)
\end{equation*}
so that $\det^{L^2}(P^{-1/2}QP^{-1/2})= \det^{L^2}(P^{-1}Q)$.

Let now $\tilde{\Pi}$ be the $L^2$-orthogonal projection in $L^2([0, t], \gamma^*TM)$ onto $\mathrm{ker}(P+\mathcal{R}_\gamma)$ and set $\tilde{Q} := P + \mathcal{R}_\gamma + \tilde{\Pi}$. We claim that $\det\nolimits_\zeta(\tilde{Q}) = \det\nolimits_\zeta(Q)$. To see this, notice first that
\begin{equation*}
  P+ \mathcal{R}_\gamma + \tilde{\Pi} = (P+\mathcal{R}_\gamma+\Pi)(\id + W),
\end{equation*}
where $W = (P+\mathcal{R}_\gamma + \Pi)^{-1} (\tilde{\Pi} - \Pi)$, which is trace-class. Now with respect to the orthogonal splitting of $L^2([0, 1], \gamma^*TM)$ into $\mathrm{ker}(P+\mathcal{R}_\gamma)$ and its orthogonal complement, the operators in question are given by
\begin{equation*}
  \Pi \,\hat{=}\, \begin{pmatrix} \id & * \\ 0 & 0\end{pmatrix} ~~~~~\tilde{\Pi} \,\hat{=}\, \begin{pmatrix} \id & 0 \\ 0 & 0\end{pmatrix} ~~~~~ P+\mathcal{R}_\gamma + \Pi\,\hat{=}\, \begin{pmatrix} \id & * \\ 0 & P+\mathcal{R}_\gamma\end{pmatrix}.
\end{equation*}
Therefore $W$ is upper triangular with respect to the splitting, hence quasi-nilpotent so that $\det(\id+W) = 1$. Thus by Prop.~\ref{PropScott}, we have
\begin{equation*}
  \det\nolimits_\zeta(\tilde{Q}) = \det\nolimits_\zeta(Q)\det(\id + W) = \det\nolimits_\zeta(Q).
\end{equation*}
Clearly, the spectrum of $\tilde{Q}$ is the same as the spectrum of $P+\mathcal{R}_\gamma$ except that the $k$-fold eigenvalue zero is replaced by $k$ times the eigenvalue one. Hence $\zeta_{\tilde{Q}}(z) = \zeta_{P + \mathcal{R}_\gamma}(z) + k$ and $\det\nolimits_\zeta(\tilde{Q}) =  \det\nolimits_\zeta^\prime(P+\mathcal{R}_\gamma)$.
By Prop.~\ref{PropScott} and Example~\ref{ExampleZetaLaplacian}, we therefore have
\begin{equation*}
\det\bigl(\nabla^2 S|_{N_{{\gamma}} \Gamma_{xy}^{\mathrm{min}}}\bigr)^{1/2}\det (d \id) = \det\nolimits^{L^2}\bigl(P^{-1}\tilde{Q}\bigr)^{1/2} = \frac{\det\nolimits_\zeta(\tilde{Q})^{1/2}}{\det\nolimits_\zeta(P)^{1/2}} = \frac{\det\nolimits_\zeta^\prime(-\nabla_s^2+\mathcal{R}_\gamma)^{1/2}}{\det\nolimits_\zeta^\prime(-\nabla_s^2)^{1/2}}.
\end{equation*}
Plugging this into \eqref{SecondLastFormulaPhiZero} gives the result.
\end{proof}

\section{Ordinary Differential Equations and the Gel'fand-Yaglom Theorem} \label{SectionGelfandYaglom}

It is a well-known fact from Riemannian geometry that for a geodesic $\gamma \in \Gamma_{xy}^{\min}$, we have
\begin{equation} \label{ExpODE}
  s \,d\exp_x|_{s\dot{\gamma}_{xy}(0)} = J(s),
\end{equation}
for each $s \in [0, 1]$ where $J(s) \in \mathrm{Hom}(T_{\gamma(0)}M, T_{\gamma(s)}M)$ is the solution to the {\em Jacobi equation}
\begin{equation} \label{JacobiEquation}
  \nabla_s^2 J(s) = \mathcal{R}_\gamma(s) J(s), ~~~~~ J(0) = 0, ~~\nabla_s J(0) = \id.
\end{equation}
see Corollary~1.12.5 in \cite{klingenberg} or Thm.~II.7.1 in \cite{ChavelRiemannian}. Hence the Jacobian of the exponential map defined in \eqref{JacobianOfExponentialMap} is given by $J(x, y) = \det(J(1))$.
Using our results above, we therefore obtain a way to calculate infinite-dimensional determinants by solving an ordinary differential equation. 

\begin{theorem}[Gel'fand-Yaglom]\label{ThmGelfandYaglom}
Let $V_i \in C^\infty([0, t], \R^{n\times n})$, $i=1, 2$ be functions with values in symmetric matrices and consider the differential operators
\begin{equation*}
  P_i := - \frac{\dd^2}{\dd s^2} + V_i.
\end{equation*}
Assume that all eigenvalues of $P_1$ and $P_2$ are positive. Then we have
\begin{equation*}
  \frac{\det\nolimits_\zeta(P_2)}{\det\nolimits_\zeta(P_1)} = \frac{\det\bigl(J_2(t)\bigr)}{\det\bigl(J_1(t)\bigr)},
\end{equation*}
where the $J_i(s)$ are the unique matrix-valued solutions of
\begin{equation*}
  {J}^{\prime\prime}_i(s) = V_i(s)J_i(s), ~~~~~~~~ J_i(0) = 0, ~~~J_i^{\prime}(0) = \id.
\end{equation*}
\end{theorem}

It seems that the name of the theorem stems from an older result by Gel'fand and Yaglom \cite{GelfandYaglom}, who express the expectation value of certain Wiener functionals as the solution to an ordinary differential equation, but without mentioning zeta determinants. A proof of Thm.~\ref{ThmGelfandYaglom} can be found in \cite{KirstenDeterminants} or \cite{KirstenMcKane} for the scalar case (i.e.\ $m=1$), using contour integrals. As we demonstrate below, Thm.~\ref{ThmGelfandYaglom} combined with Prop.~\ref{PropEnergyAsZetaQuotient} enables a different proof of the identity
\begin{equation*}
   \det\bigl(\nabla^2 S|_{\gamma_{xy}}\bigr) = J(x, y)
\end{equation*}
that gets away without having to calculate the messy term $\Upsilon_{\tau}(\gamma)$. However, this works only in the non-degenerate case. Furthermore, it turns out that the results obtained with our methods (Corollary \ref{CorollaryJacobian} and \ref{PropEnergyAsZetaQuotient}) suffice to prove Thm.\ \ref{ThmGelfandYaglom}.

\begin{proof}[of Corollary~\ref{CorollaryJacobian}, using Thm.~\ref{ThmGelfandYaglom}]
The vector bundle  $\gamma_{xy}^*TM$ over $[0, 1]$ has a canonical trivialization using parallel transport along $\gamma_{xy}$, so that Thm.~\ref{ThmGelfandYaglom} is applicable. In this local trivialization, set $V_1(s) \equiv 0$ and $V_2(s) = \mathcal{R}_{\gamma_{xy}}(s)$, the Jacobi endomorphism \eqref{JacobiEndomorphism} along $\gamma_{xy}$. Then use Thm.~\ref{ThmGelfandYaglom} with $P_1 = -\nabla_s^2$ and $P_2 = -\nabla_s^2 + \mathcal{R}_{\gamma_{xy}}$, the Jacobi operator. Clearly, $P_1$ has only positive eigenvalues, and since $(x, y) \in M \bowtie M$, $P_2$ has only positive eigenvalues as well (compare Thm.~15.1 in \cite{MilnorMorseTheory}).

Now $J_1(s) = s\,\id$ so that $\det(J_1(1)) = 1$. On the other hand, by \eqref{ExpODE}, we have $\det(J_2(1)) = J(x, y)$. Therefore,
\begin{equation*}
  \det\bigl(\nabla^2 S|_{\gamma_{xy}}\bigr) = \frac{\det\nolimits_\zeta(-\nabla_s^2 + \mathcal{R}_{\gamma_{xy}})}{\det\nolimits_\zeta(-\nabla_s^2)} = \frac{\det\bigl(J_2(1)\bigr)}{\det\bigl(J_1(1)\bigr)} = \frac{J(x, y)}{1},
\end{equation*} 
where we first used Prop.~\ref{PropEnergyAsZetaQuotient} and then Thm.\ \ref{ThmGelfandYaglom}.
\end{proof}

\begin{proof}[of Thm.~\ref{ThmGelfandYaglom}, using Corollary~\ref{CorollaryJacobian}]
Since we only calculate the ratio, we may assume $V_1 \equiv 0$. Now given a smooth function $V := V_2$ with values in symmetric $(n\times n)$-matrices, define on $M = \R \times \R^n$ (equipped with coordinates $s, x^1, \dots, x^n$) a Riemannian metric as follows. Choose neighborhoods $U_1$ and $U_2$ of $[0, t] \times \{0\}$ in $M$ such that $\overline{U_1} \subset U_2$. On $U_1$ set
\begin{equation*}
  g_{ss}(s, x) = 1 + V_{ij}(s) x^i x^j, ~~~~~ g_{sj}(s, x) = 0, ~~~~~~ g_{ij}(s, x) = \delta_{ij},
\end{equation*}
where $1 \leq i, j \leq n$ and $V_{ij}(s)$ are the entries of $V(s)$; on the complement on $U_2$, set $g_{ss} = 1$, $g_{sj} = 0$, $g_{ij} = \delta_{ij}$; on $U_2 \setminus U_1$, choose a smooth interpolation between the two metrics. One can choose the open sets and the interpolation in such a way that the resulting metric is non-degenerate; then $M$ becomes a complete Riemannian manifold. 

The curve $\gamma(s) := (s, 0, \dots, 0)$ is a geodesic from $x := (0, \dots, 0)$ to $y := (t, 0, \dots, 0)$, because all Christoffel symbols vanish at points in $[0, t] \times \{0\}$, as is easy to calculate. It is the unique shortest geodesic between $x$ and $y$ if and only if the Jacobi operator $-\nabla_s^2 + \mathcal{R}_\gamma$ on $[0, t]$ has only positive eigenvalues (see \cite[Thm~15.1]{MilnorMorseTheory}), which we assume from now on. On the other hand, one can easily compute that the Jacobi endomorphism \eqref{JacobiEndomorphism} is explicitly given by
\begin{equation} \label{FormulaOfRSeparate}
  \mathcal{R}_\gamma(s) = \begin{pmatrix} 1 & 0 \\ 0 & V(s)\end{pmatrix},
\end{equation}
so that by \eqref{ExpODE}, the differential of the exponential map is given by
\begin{equation*}
  d \exp_x|_{s \dot{\gamma}(0)} = \frac{1}{s}\begin{pmatrix}
  1 & 0 \\ 0 & J_2(s)
  \end{pmatrix},
\end{equation*}
where $J_2(s)$ is the unique matrix solution of
\begin{equation*}
  {J_2}^{\prime\prime}(s) = V(s) J_2(s), ~~~~~~~~~~~ J_2(0) = 0, ~~J_2^{\prime}(0) = \id.
\end{equation*}
The shortest geodesic travelling from $x$ to $y$ in time one, on the other hand, is given by $\gamma_{xy}(s) = \gamma(st)$. Hence
\begin{equation*}
  J(x, y) = \det\bigl( d \exp_x|_{ \dot{\gamma}_{xy}(0)}\bigr) = \det\bigl( d \exp_x|_{ t\dot{\gamma}(0)}\bigr) = \frac{\det\bigl(J_2(t)\bigr)}{t^{n+1}} = \frac{\det\bigl(J_2(t)\bigr)}{\det\bigl(J_1(t)\bigr)},
\end{equation*}
where $J_1 = t \,\id$ is the matrix solution of the equation $J_1^{\prime\prime}(t) = 0$ with initial conditions $J_1(0) = 0$, $J^\prime_1(0) = \id$.
By Prop.~\ref{PropEnergyAsZetaQuotient} and Corollary~\ref{CorollaryJacobian}, we therefore have
\begin{equation*}
  \frac{\det\nolimits_\zeta(-\nabla_s^2 + \mathcal{R}_{\gamma})}{\det\nolimits_\zeta(-\nabla_s^2)} = \frac{\det\nolimits_\zeta(-\nabla_s^2 + \mathcal{R}_{\gamma_{xy}})}{\det\nolimits_\zeta(-\nabla_s^2)} = J(x, y) = \frac{\det\bigl(J_2(t)\bigr)}{\det\bigl(J_1(t)\bigr)},
\end{equation*}
where we also used that the quotient on the left hand side does not depend on see as is easy to verify by considering the eigenvalues. Finally, because of \eqref{FormulaOfRSeparate}, the bundle separates into the direction tangent to $\dot{\gamma}$ and the orthogonal directions, so we obtain 
\begin{equation*}
  \frac{\det\bigl(J_2(t)\bigr)}{\det\bigl(J_1(t)\bigr)} = \frac{\det\nolimits_\zeta(-\nabla_s^2 + \mathcal{R}_{\gamma})}{\det\nolimits_\zeta(-\nabla_s^2)} = \frac{\det\nolimits_\zeta(P_2) \det\nolimits_\zeta(-\partial_{s}^2)}{\det\nolimits_\zeta(P_1)\det\nolimits_\zeta(-\partial_{s}^2)} = \frac{\det\nolimits_\zeta(P_2)}{\det\nolimits_\zeta(P_1)}.
\end{equation*}
This finishes the proof of Thm.~\ref{ThmGelfandYaglom}.
\end{proof}

\begin{remark}
Of course, in the formulation of Thm.~\ref{ThmGelfandYaglom}, one could use the Fredholm determinant of $P_1^{-1}P_2$ instead of the quotient of the zeta determinants. This way, one would get away without having to use Prop.~\ref{PropEnergyAsZetaQuotient}. That is, Thm.~\ref{ThmGelfandYaglom} can also be written as a theorem about usual Fredholm determinants.
\end{remark}
 
There is a Gel'fand-Yaglom-type theorem for the degenerate case too. As in Section~3 of \cite{KirstenMcKane}, one proves the following result.

\begin{theorem}[Degenerate Gel'fand-Yaglom]\label{ThmGelfandYaglomDegenerate}
With notations as in Thm.~\ref{ThmGelfandYaglom}, assume that $P_2$ is a positive operator, but that $P_1$ has the eigenvalue zero. Then we have
\begin{equation*}
  \frac{\det\nolimits_\zeta^\prime(P_2)}{\det\nolimits_\zeta(P_1)} = \frac{\det\left(\int_0^1 J_2(s)^*J_2(s) \dd s\right)}{\det\bigl(J_1(1)\bigr) \det\bigl(J_2^\prime(1)\bigr)}.
\end{equation*}
\end{theorem}

This can be used to prove the following formula for the lowest term in the heat expansion, which only depends on the Riemannian exponential map.

\begin{theorem} \label{ThmDegenerateCaseJacobi}
  For $x, y \in M$, set $S_{xy} := \{\dot{\gamma}(0) \mid \gamma \in \Gamma_{xy}^{\min}\} \subset T_x M$. Under the assumptions of Thm.~\ref{ThmL2Picture}, $S_{xy}$ is a $k$-dimensional submanifold of $T_x M$ and we have
  \begin{equation}
    \lim_{t \rightarrow 0} \,(4 \pi t)^{k/2} \frac{p_t^L(x, y)}{\e_t(x, y)} = \int_{S_{xy}} [\gamma_v\|_0^1]^{-1}\det\bigl(J^\prime(1) \bigr)^{1/2} \dd v,
  \end{equation}
  where for $v \in S_{xy}$, $\gamma_v$ is defined by $\gamma_v(s) = \exp_x(sv)$, $J(s)$ is given by \eqref{ExpODE} and we integrate with respect to the submanifold measure induced on $S_{xy}$ by the metric on $T_xM$.
\end{theorem}

Notice that $J(s)$ depends on the underlying geodesic $\gamma$, even though this is not reflected in the notation.

\begin{remark}
The formula of Thm.~\ref{ThmDegenerateCaseJacobi} should be compared with the formula
\begin{equation} \label{NonDegenerateJacobiFormulation}
\lim_{t \rightarrow 0}~\frac{p_t^L(x, y)}{\e_t(x, y)} = [\gamma\|_0^1]^{-1}\det\bigl(J(1)\bigr)^{-1/2},
\end{equation}
which holds in the case that $(x, y) \in M \bowtie M$, by \eqref{PhiKnotLimit} and \eqref{ExpODE}.
\end{remark}

\begin{proof}
Using Thm.~\ref{ThmGelfandYaglomDegenerate} on the formula from Thm.~\ref{ThmL2Picture} with $P_2 = -\nabla_s^2 + \mathcal{R}_\gamma$ and $P_1 = - \nabla_s^2$, we obtain
\begin{equation*}
\int_{\Gamma_{xy}^{\min}} [\gamma\|_0^1]^{-1} \frac{\det\nolimits_\zeta\bigl(-\nabla_s^2 \bigr)^{1/2}}{\det\nolimits_\zeta^\prime\bigl(-\nabla_s^2 + \mathcal{R}_\gamma \bigr)^{1/2}} \dd^{L^2} \gamma
= \int_{\Gamma_{xy}^{\min}} [\gamma\|_0^1]^{-1} \frac{\det\bigl(J^\prime(1) \bigr)^{1/2}}{\det\left(\int_0^1 J(s)^*J(s) \dd s\right)^{1/2}} \dd^{L^2} \gamma,
\end{equation*}
since we have $J_1(s) = s\id$, hence $\det(J_1(1)) = 1$, and $J_2(s) = J(s)$, given by \eqref{ExpODE}. Define the map
\begin{equation*}
  \phi: S_{xy} \longrightarrow \Gamma_{xy}^{\min}, ~~~~~~~~ v \longmapsto \gamma_v,
\end{equation*}
where $\gamma_v(s) = \exp_x(sv)$. Then by the transformation formula, the integral above is given by
\begin{equation} \label{FormulaBla}
\int_{S_{xy}} [\gamma_v\|_0^1]^{-1} \frac{\det\bigl(J^\prime(1) \bigr)^{1/2}}{\det\left(\int_0^1 J(s)^*J(s) \dd s\right)^{1/2}} \det\bigl(d\phi|_v)  \dd v
\end{equation}

Fix $v \in S_{xy}$. For an orthonormal basis $e_1, \dots, e_n$ of $T_xM$, let $X_1, \dots, X_n$ be the Jacobi fields along $\gamma_v$ with $\nabla_s X_j(0) = e_j$. Then $J_1(s) = (X_1(s), \dots, X_n(s))$ and
\begin{equation*}
  \det\left(\int_0^1 J_1(s)^*J_1(s) \dd s\right) = \det\left(\int_0^1 \Bigl(\< X_i(s), X_j(s)\>\Bigr)_{1 \leq ij \leq n} \dd s\right) = \det\Bigl( (X_i, X_j)_{L^2}\Bigr)_{1 \leq ij \leq n}.
\end{equation*}
Similarly,
\begin{equation*}
  \det\bigl(d\phi|_v\bigr) = \det\Bigl( (d\phi|_v e_i, d\phi|_v e_j)_{L^2}\Bigr)_{1 \leq ij \leq n} = \det\Bigl( (X_i, X_j)_{L^2}\Bigr)_{1 \leq ij \leq n}.
\end{equation*}
Therefore,  two of the determinants in \eqref{FormulaBla} cancel and we are left with the integrand from the theorem.

\end{proof}

    \bibliography{Literatur}

\begin{thebibliography}{Lud16b}

\bibitem[AD99]{anderssondriver}
Lars Anderson and Bruce Driver.
\newblock Finite dimensional approximations to wiener measure and path integral
  formulas on manifolds.
\newblock {\em Journal of Functional Analysis}, 165:430--498, 1999.

\bibitem[B{\"a}r12]{baerrenormalization}
Christian B{\"a}r.
\newblock Renormalized integrals and a path integral formula for the heat
  kernel on a manifold.
\newblock {\em Contemporary Mathematics}, 584:179--197, 2012.

\bibitem[BGV04]{bgv}
Nicole Berline, Ezra Getzler, and Michele Vergne.
\newblock {\em Heat Kernels and Dirac Operators}.
\newblock Springer, Berlin, Heidelberg, New York, 2004.

\bibitem[BP10]{baerpfaeffle}
Christian B{\"a}r and Frank Pf{\"a}ffle.
\newblock Asymptotic heat kernel expansion in the semi-classical limit.
\newblock {\em Communications in Mathematical Physics}, 294:731--744, 2010.

\bibitem[Cha84]{ChavelEigenvalues}
I.~Chavel.
\newblock {\em Eigenvalues in Riemannian Geometry}.
\newblock Academic Press Inc., Orlando, San Diego, New York, London, 1984.

\bibitem[Cha06]{ChavelRiemannian}
I.~Chavel.
\newblock {\em Riemannian Geometry}.
\newblock Cambridge University Press, Cambridge, 2006.

\bibitem[Con94]{ConwayFunctionalAnalysis}
J.~B. Conway.
\newblock {\em A Course in Functional Analysis}.
\newblock Springer, Berlin, Heidelberg, New York, 1994.

\bibitem[Eul40]{EulerBaslerProblem}
L.~Euler.
\newblock De summis serierum reciprocarum.
\newblock {\em Commentarii academiae scientiarum Petropolitanae}, 7:123--134,
  1740.

\bibitem[FB05]{freitagbusam}
Eberhard Freitag and Rolf Busam.
\newblock {\em Complex Analysis}.
\newblock Springer, Berlin, Heidelberg, New York, 2005.

\bibitem[FH65]{FeynmanHibbs}
R.~Feynman and A.~R. Hibbs.
\newblock {\em Quantum Mechanics and Path Integrals}.
\newblock McGraw-Hill Companies, New York, St. Louis, San Francisco, 1965.

\bibitem[Gil95]{gilkey95}
Peter Gilkey.
\newblock {\em Invariance Theory, the Heat Equation, and the Atiyah-Singer
  Index Theorem}.
\newblock CRC Press, Boca Raton, 1995.

\bibitem[GY60]{GelfandYaglom}
I.~M. Gel'fand and A.~M. Yaglom.
\newblock Integration in functional spaces and its applications in quantum
  physics.
\newblock {\em Journal of Mathematical Physics}, 1:48--69, 1960.

\bibitem[Haw77]{Hawking}
S.~W. Hawking.
\newblock Zeta function regularization of path integrals in curved spacetime.
\newblock {\em Commun. Math. Phys.}, 55:133--148, 1977.

\bibitem[Hsu02]{hsu}
Elton~P. Hsu.
\newblock {\em Stochastic Analysis on Manifolds}.
\newblock American Mathematical Society, Providence, 2002.

\bibitem[Kir10]{KirstenDeterminants}
K.~Kirsten.
\newblock Functional determinants in higher dimensions using contour integrals.
\newblock {\em A Window Into Zeta and Modular Physics}, 57:307--328, 2010.

\bibitem[Kle09]{Kleinert}
Hagen Kleinert.
\newblock {\em Path Integrals in Quantum Mechanics, Statistics, Polymer
  Physics, and Financial Markets}.
\newblock World Scientific, Singapore, 5th edition, 2009.

\bibitem[Kli95]{klingenberg}
W.~A. Klingenberg.
\newblock {\em Riemannian Geometry}.
\newblock de Gruyther, Boston, 1995.

\bibitem[KM03]{KirstenMcKane}
K.~Kirsten and A.~J. McKane.
\newblock Functional determinants by contour integration methods.
\newblock {\em Annals of Physics}, 308:502--527, 2003.

\bibitem[Lae13]{Laetsch}
Thomas Laetsch.
\newblock An approximation to wiener measure and quantization of the
  hamiltonian on manifolds with non-positive sectional curvature.
\newblock {\em Journal of Functional Analysis}, 265(8):1667--1727, 2013.

\bibitem[Li17]{LiPathIntegrals}
Zhehua Li.
\newblock A finite dimensional approximation to pinned wiener measure on some
  symmetric spaces.
\newblock arXiv:1702.06747, 2017.

\bibitem[Lim07]{LimPathintegrals}
A.~Lim.
\newblock Path integrals on a compact manifold with non-negative curvature.
\newblock {\em Reviews in Mathematical Physics}, 19(9):967--1044, 2007.

\bibitem[Lud16a]{ludewigThesis}
Matthias Ludewig.
\newblock {\em Path integrals on manifolds with boundary and their asymptotic
  expansions.}
\newblock PhD thesis, Universit\"at Potsdam, 2016.

\bibitem[Lud16b]{LudewigStrongAsymptotics}
Matthias Ludewig.
\newblock Strong short time asymptotics and convolution approximation of the
  heat kernel.
\newblock {\em arxiv:1607.05152}, 2016.

\bibitem[Lud17]{ludewigBoundary}
Matthias Ludewig.
\newblock Path integrals on manifolds with boundary.
\newblock {\em Comm. Math. Phys.}, 354(2):621--640, 2017.

\bibitem[Mc75]{molchanov}
S.~A. Mol\v~canov.
\newblock Diffusion processes, and {R}iemannian geometry.
\newblock {\em Uspehi Mat. Nauk}, 30(1(181)):3--59, 1975.

\bibitem[Mil63]{MilnorMorseTheory}
J.~Milnor.
\newblock {\em Morse Theory}.
\newblock Princeton University Press, Princeton, 1963.

\bibitem[PPV11]{TheMotionofPointParticlesinCurvedSpacetime}
E.~Poisson, A.~Pound, and I.~Vega.
\newblock The motion of point particles in curved spacetime.
\newblock {\em Living Rev. Relativity}, 14(7), 2011.

\bibitem[Sch13]{BernoulliFourierSeries}
E.~E. Scheufens.
\newblock Bernoulli polynomials, fourier series and zeta numbers.
\newblock {\em International Journal of Pure and Applied Mathematics},
  88(1):65--75, 2013.

\bibitem[Sco02]{ScottDeterminants}
S.~Scott.
\newblock Zeta determinants on manifolds with boundary.
\newblock {\em Journal of Functional Analysis}, 192(1):112--185, 2002.

\bibitem[Sim77]{SimonDeterminants}
Barry Simon.
\newblock Notes on infinite determinants of hilbert space operators.
\newblock {\em Advances in Mathematics}, 24:244--273, 1977.

\bibitem[Son94]{SondowRiemannValues}
J.~Sondow.
\newblock Analytic continuation of riemann's zeta zunction and values at
  negative integers.
\newblock {\em Proc. Amer. Math. Soc.}, 120(2):421--424, 1994.

\bibitem[Wit99]{WittenDiracOperators}
E.~Witten.
\newblock Index of dirac operators.
\newblock In P.~Deligne et~al., editor, {\em Quantum Fields and Strings: A
  Course for Mathematicians}. American Mathematical Society, Providence, 1999.

\end{thebibliography}
 
\end{document}